\documentclass[10pt]{article}
\usepackage{latexsym}
\usepackage{amsmath,amsthm}
\usepackage{amssymb,amsfonts}
\usepackage{color}
\usepackage{graphicx}
\usepackage{epsfig}

\def\e{\epsilon}
\def\R{{ I\!\!R}}

\def\II{{\rm I\kern-0.5exI}}
\def\III{{\rm I\kern-0.5exI\kern-0.5exI}}

\numberwithin{equation}{section}
\newtheorem{theorem}{Theorem}[section]
\newtheorem{remark}[theorem]{Remark}
\newtheorem{lemma}[theorem]{Lemma}

\newtheorem{proposition}[theorem]{Proposition}
\newtheorem{corollary}[theorem]{Corollary}

\begin{document}
\title{Two phase Stefan-type problem: Regularization near initial data by phase dynamics.}
\author{Sunhi Choi\thanks{Department of Mathematics, University of Arizona} and Inwon Kim\thanks{Department of Mathematics, UCLA.  I.K. is partially supported by NSF 0970072 }}
\maketitle
\begin{abstract}
In this paper we investigate the regularizing behavior of two-phase Stefan problem near initial data. The main step in the analysis is to establish that in any given scale, the scaled solution is very close to a Lipschitz profile in space-time. We introduce a new decomposition argument to generalize the preceding ones (\cite{cjk1}-\cite{cjk2} and \cite{ck}) on one-phase free boundary problems.
\end{abstract}
\section{Introduction}

 Consider $u_0(x):B_R(0)\to \R$ with $R>>1$ and $u_0\geq -1$, $|\{u_0=0\}|=0$ and $u_0(x)=-1$ on $\partial B_R(0)$.
 (See Figure 1.)
\begin{figure}
\center{\epsfig{file=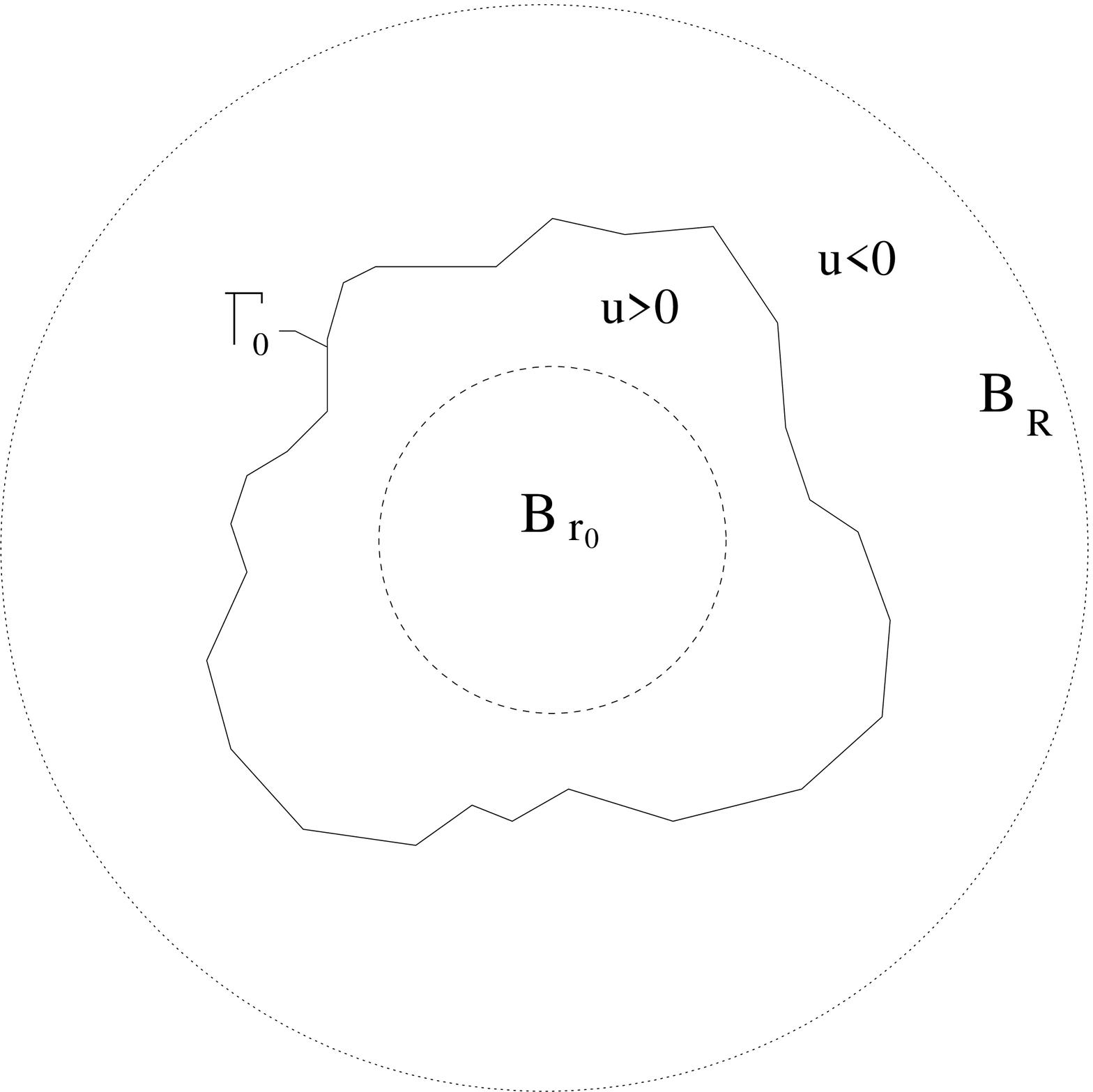,height=2.0in}} \center{Figure
1: Initial setting of the problem}
\end{figure}
The two-phase Stefan problem can be formally written as
$$
\left\{\begin{array} {lll}
u_t-\Delta u = 0 &\hbox{ in } & \{u>0\} \cup \{u<0\}\\ \\
\displaystyle{\frac{u_t}{|Du^+|}}=|Du^+|-|Du^-| &\hbox{ on } & \partial\{u>0\}\\ \\
u(\cdot,0)=u_0&&\\ \\
u = -1 &\hbox{ on }& \partial B_R(0)
\end{array}\right.\leqno (ST2)
$$
where $u^+$ and $u^-$ respectively denotes the positive and negative parts of $u$, i.e,
$$
u^+:=\max(u,0)\hbox{ and } u^-:=-\min (u,0).
$$
The classical Stefan problem describes the phase transition between
solid/liquid or liquid/liquid interface (see \cite{m} and also
\cite{opr}.) In our setting, we consider a bounded domain
$\Omega_0\subset B_R(0)$ and the initial data $u_0(x)$ such that
$\{u_0>0\} = \Omega_0$ and $\{u_0<0\} = \R^n-\Omega_0$. To avoid
complicity at the infinity,  we consider the problem in the domain
$Q = B_R(0)\times [0,\infty)$, with Dirichlet condition
$$
u = f(x,t)<0\hbox{ on }\partial B_R(0),
$$
where $f(x,t)$ is smooth. In (ST2) we have set $f=-1$ for
simplicity. Since our initial data will be only locally H\"{o}lder
continuous, we employ the notion of viscosity solutions to discuss
the evolution of the problem. Viscosity solutions for (ST2) is
originally introduced by \cite{acs1} (also see \cite{cs}). As for
existence and uniqueness of viscosity solutions,  we refer to
\cite{kp}.

\vspace{5 pt}
Note  that the second condition of $(ST)$ states
that the normal velocity $V_{x,t}$ at each free boundary point
$(x,t)\in\partial\{u>0\}$ is given by
$$
V_{x,t}=|Du^+|-|Du^-|(x,t)=(Du^+(x,t)-Du^-(x,t))\cdot\nu_{x,t},
$$
where $\nu_{x,t}$ denotes the spatial unit normal vector of
$\partial\{u>0\}$ at $(x,t)$, pointing inward with respect to the \textit{positive phase} $\{u>0\}$.

\vspace{10 pt}

In this paper we investigate the regularizing behavior of the \textit{ free
boundary} $\partial\{u>0\}$. Our main result states that when $\Gamma_0$ has  no sharp corner,
then the free boundary immediately regularizes after $t=0$, and
stays regular for a small amount of time. Note that, in general, after some
time the free boundary may move away from its initial profile and
develop singularities by topological changes, such as merging of two
boundary parts. Whether this happens with star-shaped initial data
is an open question (see Remark~\ref{star-shaped}.)

\vspace{5 pt}

The well-known results of \cite{acs1}-\cite{acs2} states that if a
solution as well as its free boundary of $(ST2)$ stays close to a
locally Lipschitz profile in a unit space-time neighborhood, then
the solution is indeed smooth in a smaller neighborhood.  Hence the
main step in our analysis is to prove that the free boundary
$\partial\{u>0\}$ stays close to a locally Lipschitz profile over a
unit time interval. Indeed proving this step has been the main challenge in the
previous work of the authors and Jerison (\cite{cjk1}, \cite{cjk2}, \cite{ck}) on the
study of one-phase free boundary problems. Once this step is
established, using the fact that $u$ is a caloric function in {\it almost
Lipschitz } domain, we will have some control over the behavior of
re-scaled solutions following the arguments in \cite{cjk1}.  Then
the appropriate modification of iteration arguments taken in
\cite{acs1}-\cite{acs2}  applies to derive further regularity
results (see section 5).  In extension of the ideas from one-phase case to our setting, the main obstacle lies in the competition between
fluxes of positive and negative phase: to overcome this, we introduce a new decomposition procedure which we explain below.

\vspace{10pt}

Before discussing our result in detail,  let us introduce precise conditions on the initial data.

\begin{itemize}
\item[(I-a)]$\Omega_0$ and $u_0$ are star-shaped  with respect
 to a ball $B_{r_0}(0)\subset \Omega_0$.
\end{itemize}
Observe that then the  {\it Lipschitz constant} $L$ of $\partial
\Omega_0$ is determined by $r_0$ and $d_0$, where $d_0:=
\sup\{d(x,B_{r_0}(0): x\in\partial\Omega_0\}$: i.e.,  there exist $h
= h(r_0)$ and $L=L(\frac{r_0}{d_0})$ such that for any
$x_0\in\partial\Omega_0$, after rotation of coordinates one can
represent
\begin{equation}\label{initialLipschitz}
B_h(x_0)\cap\Omega_0 = \{(x',x_n): x'\in\R^{n-1}, x_n\leq f(x)\}\hbox{ with } Lip f \leq L.
\end{equation}
For simplicity of the presentation we set $h = 1$.

\vspace{10pt}

 For a locally Lipschitz
domain such as $\Omega_0$, there exist growth rates
$0<\beta<1<\alpha$ such that the following holds: Let $H$ be a
positive harmonic function in $\Omega_0 \cap B_2(x)$, $x \in
\partial \Omega_0$, with Dirichlet condition on $\partial \Omega_0 \cap
B_2(x)$, and with value $1$ at $x-e_n$. (Here let $e_n$ be the
direction of the axis for the Lipschitz graph near $x$.) Then for $
x-se_n \in \Omega_0 \cap B_1(x)$
\begin{equation}\label{growth}
 s^{\alpha}\leq H(x -se_n) \leq s^{\beta}.
\end{equation}
Now we precisely describe the range of the Lipschitz constant $L$ of
$\Omega_0$.
\begin{itemize}
\item[(I-b)]  $L<L_n$ for a sufficiently small dimensional constant  $L_n$  so  that
$$
5/6 \leq \beta < \alpha \leq 7/6.
$$
\end{itemize}
The remaining conditions are on the regularity of $u_0$.
\begin{itemize}
\item[(I-c)] $-N_0 \leq \Delta u_0 \leq N_0$ in $\Omega_0 \cup (B_R(0) -\Omega_0)$,\\
\item[(I-d)] For $x \in \partial \Omega_0$, $e_n=x/|x|$ and small $s>0$ (for $0<s<1/10$),
$$|D u_0(x \pm se_n)| \geq C s^{\alpha-1}.$$
\end{itemize}
Note that (I-c) and (I-d) holds for $u_0$ which is smooth in its positive and negative phases and is harmonic near the initial free boundary: i.e., $-\Delta u_0 = 0$ in the set
$(\{u_0>0\} \cup\{u_0<0\}) \cap\{x: d(x,\partial\Omega_0) \leq 1\}$.
\vspace{ 5 pt}

For a function $u(x,t):\R^n\times [0,\infty) \to \R$, let us denote
$$
\Omega(u):=\{u>0\}, \quad \Omega_t(u):=\{u(\cdot, t)>0\}
$$
and
$$  \Gamma(u):=\partial\{u>0\}, \quad  \Gamma_t(u):=\partial\{u(\cdot, t)>0\}.
$$
Since $\Gamma_0=\partial\{u(\cdot,0)<0\}$ in our setting, the property is preserved for later times,
i.e.,
$$
\Gamma_t(u)=\partial\{u(\cdot,t)>0\} = \partial \{u(\cdot,t)<0\}\hbox{ for all } t>0
$$
(see \cite{rb}, \cite{gz}, and \cite{kp}).

In the one-phase case the problem can be written as follows:
$$
\left\{\begin{array} {lll}
u_t-\Delta u = 0 &\hbox{ in } & \{u>0\} \\ \\
\displaystyle{\frac{u_t}{|Du|}}=|Du| &\hbox{ on } & \partial\{u>0\}\\ \\
u(\cdot,0)=u_0^+.
\end{array}\right.\leqno (ST1)
$$
In \cite{ck} the following has been proved for (ST1).

\begin{theorem}\label{thm:ckmain}([CK] Theorem 0.1.)
 Suppose $u$ is a solution of $(ST1)$ in $B_2(0)\times [0,1]$, $0 \in \Gamma_0(u)$, with the initial data $u_0\geq 0$
satisfying (I-b), (I-c) and (I-d) in $B_2(0)$. Suppose
$u_0(-e_n)=1$.
 If $\sup_{B_2(0)\times [0,1]}u \leq M_0$, then there exists a small $s>0$ depending on $N_0$,
 $M_0$ and $n$ such that the free boundary $\Gamma_t(u)$ becomes
smooth and averages out in $B_{s}(0)$. More precisely,
\begin{itemize}
\item[(a)]  The free boundary $\Gamma_t(u)$ is $C^1$ and is a
Lipschitz graph with respect to $e_n$ with Lipschitz constant
$L'<L_n$ in $B_{s}(0) \subset \R^n$.
\\
\item[(b)] The spatial normal of $\Gamma_t(u)$  is continuous
in space and time, in $B_{s}(0)$.\\
\item[(c)] If $x\in\Gamma_0(u)\cap B_{s}(0)$ and $x+de_n\in\Gamma_t(u)\cap
B_{s}(0)$, then
$$
C^{-1} \dfrac{u(x-de_n,0)}{d} \leq |Du(x+de_n,t)|= V_{x+de_n,t}\leq
C \dfrac{u(x-de_n,0)}{d}
$$
where $C$ depends on $n$ and $M_0$. Hence
$$\dfrac{d}{t} \sim |Du(x+de_n,t)| \sim \dfrac{u(x-de_n,0)}{d}.$$
\end{itemize}
 \end{theorem}

Theorem~\ref{thm:ckmain} states that  the free boundary regularizes
in space, in a scale proportional to the distance it has traveled.
Note that the regularity results hold up to the initial time
and all the regularity assumptions are imposed only on the initial
data.

\vspace{10pt}

 Our aim in this paper is to extend the above theorem to the
two-phase case. Here the intuition is rather straightforward, based
on the previous results. There are two cases:
\begin{itemize}
\item[(a)] One of the phases has much bigger flux than the other: in this case one-phase like phenomena (regularization by the dominant phase proportional to the distance the free oundary traveled) is expected.

\item[(b)] Both phases are in balance: in this case one expects regularization due to competition between two phases, resulting in Lipschitz-like behavior over  time.
\end{itemize}

 The difficulty in making above heuristics rigorous lies in introducing a proper ``sorting" procedure
which divides the cases (a) and (b) in a given scale.  To enable such procedure, it is essential to show Harnack-type
inequalities for solutions of (ST2) in both cases, ensuring that the behavior of solutions can be localized in a proper time-space scale.

\vspace{5pt}

To state the main result, we introduce one more notation.
\vspace{5pt}

 \noindent $\bullet$  For $x_0\in\Gamma_0=\Gamma_0(u)$  and $e_n :=
x_0/|x_0|$, define
$$
R^+(x_0,d):= \dfrac{u^+(x_0-de_n,0)}{u^-(x_0+de_n,0)}, \quad
R^-(x_0,d):=\dfrac{u^-(x_0+de_n,0)}{u^+(x_0-de_n,0)}
$$
and
 $$
 t(x_0,d) := \min[\frac{d^2}{u^+(x_0-de_n,0)}, \frac{d^2}{u^-(x_0+de_n,0)}].
 $$

\vspace{5pt}

\begin{theorem}[Main Theorem] \label{maintheorem}
Suppose $u$ is a solution of (ST2) with initial data $u_0$
satisfying (Ia)-(Id) with $\Omega_0(u) \subset B_2(0) $. Then there
exists a constant $d_0$ depending on the dimension $n$ and $N_0$
such that the following holds. If $x_0\in\Gamma_0(u)$ and $d \leq
d_0$, then
 $\Gamma(u)$ is a Lipschitz graph in space-time in the region $B_{2d}(x_0)\times [t(x_0,d)/2, t(x_0,d)]$
 with $\Gamma(u)$ intersecting with $B_d(x_0)\times [t(x_0,d)/2, t(x_0,d)]$.
 Further, there exists a positive dimensional constant $M$ such that the following
 holds.
\begin{itemize}
\item[(a)] If $R^+(x_0,d) \geq M$, then
$$
 |Du^+|(x,t) \sim \dfrac{u^+(x_0-de_n,0)}{d}
$$
in $B_{2d}(x_0)\times \displaystyle{[\frac{t(x_0,d)}{2},t(x_0,d)]}$
and
$$
V_{x,t} \sim \dfrac{u^+(x_0-de_n,0)}{d}
$$
on $\displaystyle{\Gamma(u) \cap (B_{2d}(x_0)\times
[\frac{t(x_0,d)}{2},t(x_0,d)])}$.

The parallel statements hold for $u^-$ if $R^-(x_0,d)\geq M$.

\item[(b)] If $R^+(d), R^-(d)\leq M$, then
$$
 |Du^{\pm}|(x,t)
\sim  \dfrac{u^+(x_0-de_n,0)}{d} \sim \dfrac{u^-(x_0+de_n,0)}{d}
$$
 in $ B_{2d}(x_0)\times [t(x_0,d)/2,t(x_0,d)]$.

\end{itemize}
\end{theorem}

\vspace{10pt}

\begin{remark}
\begin{itemize}
\item[1.] Note that in the first case, $t(x_0,d)$ indeed is comparable to the time that $\Gamma(u)$ has
traveled from $x_0$ to $x_0 \pm de_n$, and thus we can say that the free boundary regularizes in a scale
 proportional to the distance it has traveled.
\item[2.] Our result extends to the case where the star-shaped condition (I-a) is replaced by
\begin{itemize}
\item[(I-a)'] $\Omega_0$  is locally Lipschitz with a sufficiently small Lipschitz constant.
\end{itemize}
We discuss the difference in the proof in this case, in section 6.
\end{itemize}
\end{remark}

\vspace{5pt}

Let us finish this section with an outline of the paper.
 In section 2 we introduce preliminary results and notations, to be used in the paper.
In section 3 we prove some properties on the evolution of solutions
of (ST2) with star-shaped data.  In addition to Harnack
inequalities, we show that the solution stays near the star-shaped
profile for a unit time (Lemma~\ref{star-shaped}), which in turn
yields that the solution stays very close to harmonic functions
(Lemma~\ref{regularity}). Making use of the results in section 3, in
section 4 we perform a decomposition procedure to show that for a
unit time all parts of the free boundary stay close to Lipschitz
profiles, regardless of the local dynamics between the phases. This
completes our main step in the analysis. In section 5 we describe
the procedure leading to further regularization, pointing out the
main difference between the previous results.  In section 6 we discuss a generalized proof
for the  corresponding regularization result (Theorem\ref{maintheorem2}) when the star-shapedness of the initial data (I-a) is replaced by a local version (I-a)'.

\section{Preliminary lemmas and notations}

We introduce some notations.

\vspace{5 pt}

\noindent $\bullet$ For $x \in \R^n$, denote $x=(x',x_n) \in
\R^{n-1}\times \R$ where $x_n = x\cdot e_n$.

\vspace{5 pt}

\noindent $\bullet$ Let $B_r(x)$ be the space ball of radius $r$,
centered at $x$.

\vspace{5 pt}

\noindent $\bullet$ Let $Q_r:= B_r(0)\times [-r^2,r^2]$ be the
parabolic cube and let $K_r := B_r(0)\times [-r,r]$ be the
hyperbolic cube.

\vspace{5 pt}

\noindent $\bullet$ A caloric function in $\Omega\cap Q_r$ will
denote a nonnegative solution of the heat equation, vanishing along
the lateral boundary of $\Omega$.

\vspace{5 pt}

\noindent $\bullet$ For $x_0 \in \Gamma_0$ and $e_n=x_0/|x_0|$,
define
$$
t(x_0,d) := \min[\dfrac{d^2}{u^+(x_0-de_n,0)},
\dfrac{d^2}{u^-(x_0+de_n,0)}].
$$

\vspace{5 pt}

\noindent $\bullet$ Given $\e>0$,  a function $w$ is called
$\e$-monotone in the direction $\tau$ if
 $$
 u(p+\lambda\tau) \geq u(p) \hbox{ for any } \lambda \geq \e.
 $$

\noindent $\bullet$ $W_x(\theta^x,e)$ and $W_t(\theta^t,\nu)$ with
$e\in\R^{n}$ and $\nu\in span(e_n,e_t)$ respectively denote a
spatial circular cone of aperture $2\theta^x$ and axis in the
direction of $e$, and a two-dimensional space-time cone in
$(e_n,e_t)$ plane of aperture $2\theta^t$ and axis in the direction
of $\nu$.

\vspace{5 pt}

\noindent $\bullet$ $w$ is $\e$-monotone in a cone of directions if
$w$ is $\e$-monotone in every direction in the cone.

\vspace{5 pt}

\noindent $\bullet$ $C$ is called an universal constant if it
depends only on the dimension $n$ and the regularity constant $N_0$
of $u_0$.

\vspace{5 pt}

 The first lemma is a direct consequence of the
interior Harnack inequalities proved in [C-C].

\begin{lemma} [\cite{c-c}] \label{lem:C-C}
Suppose $w(x):\R^n\to \R$ has bounded Laplacian. Then $w$ is
H\"{o}lder continuous with its constant depending on the Laplacian
bound.
\end{lemma}

\begin{lemma}[\cite{fgs1},  Theorem 3]  \label{lem:fgs84}
 Let $\Omega$ be a domain in $\R^n\times \R$ such that
 $(0,0)$ is on its lateral boundary. Suppose $\Omega$ is a $Lip^{1,1/2}$ domain, i.e.,
$$
\Omega= \{(x',x_n,t) : |x'| <1,  |x_n| < 2L, |t|<1,   x_n \leq
f(x',t)\},
$$
where $f$ satisfies $|f(x',t)-f(y',s)| \leq L(|x'-y'|+|t-s|^{1/2}.)$
If $u$ is a caloric function in $\Omega$, then there exists
$C=C(n,L)$, where $L$ is the Lipschitz constant for $\Omega$, such
that
$$
\dfrac{u(x,t)}{v(x,t)}\leq C\dfrac{u(-Le_n, 1/2)}{v(-Le_n,-1/2)}.
$$
for $(x,t)\in Q_{1/2}$.
\end{lemma}

\begin{lemma} [\cite{acs1},  Theorem 1] \label{lem:acs96}
Let $\Omega$ be a Lipschitz domain in $\R^n\times \R$, i.e.,
$$
Q_1 \cap \Omega = Q_1 \cap\{(x,t): x_n \leq f(x',t)\},
$$
where $f$ satisfies $|f(x,t)-f(y,s)| \leq L (|x-y|+|t-s|)$. Let $u$
be a caloric function in $Q_1\cap \Omega$ with
$(0,0)\in\partial\Omega$ and $u(-e_n,0) = m>0$ and $\sup_{Q_1} u
=M$. Then there exists a constant $C$, depending only on $n$, $L$,
$\frac{m}{M}$ such that
$$
u(x,t+\rho^2) \leq Cu(x,t-\rho^2)
$$
for all $(x,t)\in Q_{1/2}\cap \Omega$ and for $0\leq \rho \leq
d_{x,t}$.
\end{lemma}
\begin{lemma}  [\cite{acs1}, Lemma 5] \label{lem:almostharmonic}
Let $u$ and $\Omega$ be as in  Lemma~\ref{lem:acs96}. Then there
exist $a, \delta>0$ depending only on $n$, $L$, $\frac{m}{M}$ such
that
$$
w_+ := u+u^{1+a}\hbox{ and } w_- := u-u^{1+a}
$$
are subharmonic and superharmonic, respectively, in $
Q_{\delta}\cap\Omega\cap\{t=0\}.$
\end{lemma}

Next we state several properties of harmonic functions:

\begin{lemma} [\cite{d}] \label{lem:D}
Let $u_1,u_2$ be two nonnegative harmonic functions in a domain $D$
of $\R^n$ of the form
$$
D=\{(x',x_n)\in\R^{n-1}\times\R: |x'|<2, |x_n|<2L, x_n>f(x')\}
$$
with $f$ a Lipschitz function with constant less than $L$ and
$f(0)=0.$ Assume further that $u_1=u_2=0$ along the graph of $f$.
Then in
$$
D_{1/2}=\{|x'|<1, |x_n|<L, x_n>f(x')\}
$$
we have
$$
0<C_1\leq\frac{u_1(x',x_n)}{u_2(x',x_n)}\cdot\frac{u_2(0,L)}{u_1(0,L)}\leq
C_2
$$
with $C_1,C_2$ depending only on $L$.
 \end{lemma}

\begin{lemma} [\cite{jk}] \label{lem:JK}
Let $D$, $u_1$ and $u_2$ be as in Lemma~\ref{lem:D}. Assume further
that
$$
\dfrac{u_1(0,L/2)}{u_2(0,L/2)} = 1.
$$
Then, $u_1(x',x_n)/u_2(x',x_n)$ is H\"{o}lder continuous in
$\bar{D}_{1/2}$ for some coefficient $\alpha$, both $\alpha$ and the
$C^\alpha$ norm of $u_1/u_2$ depending only on $L$.
\end{lemma}

\begin{lemma}  [\cite{c2}] \label{lem:2.11}
 Let $u$ be as in
Lemma~\ref{lem:D}. Then there exists $c>0$ depending only on $L$
such that for $0<d<c$,
 $\frac{\partial}{\partial x_n}u(0,d)\geq 0$ and

$$ C_1\frac{u(0,d)}{d} \leq \frac{\partial u}{\partial x_n}(0,d)\leq
C_2\frac{u(0,d)}{d}$$ where $C_i=C_i(M)$.
\end{lemma}

\begin{lemma} [\cite{jk}, Lemma 4.1] \label{lem:JeKi}  Let $\Omega$ be
Lipschitz domain contained in $B_{10}(0)$. There exists a
dimensional constant $\beta_n> 0$ such that for any $\zeta \in
\partial \Omega$, $0 < 2r < 1$ and positive harmonic function $u$ in
$\Omega\cap B_{2r}(\zeta)$, if $u$ vanishes continuously on
$B_{2r}(\zeta) \cap \partial \Omega$, then for $x \in \Omega \cap
B_r(\zeta)$,
$$
 u(x) \leq C(\dfrac{|x - \zeta|}{r})^{\beta_n} {\rm
sup} \{u(y) : y \in
\partial B_{2r}(\zeta) \cap \Omega\}
$$
 where $C$ depends only on
the {\rm Lipchitz} constants of $\Omega$.
\end{lemma}

Next, we point out that we use the notion of viscosity solutions for
our investigation.  When $\{u_0=0\}$ is of zero Lebesgue measure, it
was proved in \cite{kp} that the viscosity solution of $(ST2)$ is
unique and coincides with the usual weak solutions. (See \cite{kp}
for the definition as well as other properties of viscosity
solutions.) Below we state important properties of viscosity
solutions.
\begin{lemma}\label{basic:visc}
Suppose $u$ is a viscosity solution of (ST2). Then
\begin{itemize}
\item[(a)] $u$ is caloric in its positive and negative phases.
\item[(b)]  $-u$ is also a viscosity solution of (ST2) with boundary data $-g$.
\item[(c)] $u^+=\max(u,0)$ (or $u^-=-\min(u,0)$) is a viscosity subsolution (or supersolution)  of (ST2)
with initial data $u_0^+$(or $u_0^-$).
\end{itemize}
\end{lemma}

\begin{lemma}  [Comparison principle, \cite{kp}] \label{thm:cp}Let $u,v$ be respectively viscosity sub-
and supersolutions of (ST2) in $D\times (0,T)\subset Q$ with initial
data $u_0\prec v_0$ in $D$. If
 $u\leq v$ on $\partial D$ and $u<v$ on $\partial D \cap\bar{\Omega}(u)$ for $0\leq
t< T$,
 then $u(\cdot,t)\prec v(\cdot,t)$ in $D$ for $t\in [0,T)$.
\end{lemma}

Below we state  a distance estimate for the free boundary  and
Harnack inequality for the one-phase solution $u$ of (ST1).

\begin{lemma} [\cite{ck}, Lemma 2.2] \label{good}
Let $u$ be given as in Theorem~\ref{thm:ckmain}. There exists $t_0
=t_0(N_0, M_0, n)>0$ such that if $x_0 \in \Gamma_0 $ and $t \leq
t_0$, then
\begin{equation} \label{eq:001}
\frac{1}{C} t^{1/(2 -\alpha)} \leq d(x_0, t ) \leq C t^{1/(2
-\beta)}
\end{equation}
where $\alpha $ and $\beta$ are given in (\ref{growth}),  $C$
depends on $N_0$, $M_0$ and $n$, and $d(x_0,t)$ denotes the distance
that $\Gamma$ moved  from the point $x_0$ during the time $t$, i.e.,
$$d(x_0, t) := \sup\{d: u(x_0+de_n,t)>0\}. $$
\end{lemma}

\begin{lemma} [\cite{ck}, Lemma 2.3]
Let $u$ be given as in Theorem~\ref{thm:ckmain}. There exists $d_0$
depending on $N_0$, $M_0$ and $n$ such that if $x_0 \in \Gamma_0$
and $ d \leq d_0$, then
$$u(x_0-de_n,t) \leq Cu(x_0-de_n,0) \hbox{ for } 0 \leq t\leq t(x_0,d)$$
where $C$ depends on $N_0$, $M_0$ and $n$.
\end{lemma}

The following monotonicity formula  by Alt-Caffarelli-Friedman
prevents the scenario that both phases compete with large pressure
in our problem.

\begin{lemma} [\cite{acf84}] \label{monotonicityformula}
Let $h_+$ and $h_-$ be nonnegative continuous functions in $B_1(0)$
such that $\Delta h_{\pm} \geq 0$ and $h_+ \cdot h_- =0$ in
$B_1(0)$. Then the functional
$$\phi(r) =\frac{1}{r^4} \int_{B_r(0)} \frac{|\nabla
h_+|^2}{|x|^{n-2}}dx \int_{B_r(0)} \frac{|\nabla
h_-|^2}{|x|^{n-2}}dx$$ is monotone increasing in $r$, $0 <r<1$.

\end{lemma}

\vspace{20pt}

\begin{corollary} \label{moncor}
Let $\partial \Omega_0 \subset \R^n$ be star-shaped with respect to
$B_{1}(0) \subset \Omega_0$ and suppose $B_{4/3}(0) \subset \Omega_0
\subset B_{5/3}(0) $. Let $h_+$ be the harmonic function in
$\Omega_0 - B_1(0)$ with boundary values $h_+ =0 $ on $\partial
\Omega_0$, and $h_+=1$ on $\partial B_1(0)$. Let $h_-$ be the
harmonic function in $B_2(0) -\Omega_0$ with boundary values $h_- =0
$ on $\partial \Omega_0$, and $h_-=1$ on $\partial B_2(0)$.   Then
there exists a sufficiently large dimensional constant $M>0$ such
that
$$\frac{h_+(x_0-re_n )}{r } \geq M
\hbox{ implies } \frac{h_-(x_0+re_n)}{r } \leq 1$$ for $x_0 \in
\partial \Omega_0$,  $e_n=x/|x|$ and $0 \leq r \leq 1/6$.

\end{corollary}

\begin{proof}

It follows from Lemma~\ref{monotonicityformula} since
$$
\begin{array}{lll}
& &\left(\dfrac{h_+(x_0-re_n )}{r }\cdot \dfrac{h_-(x_0+re_n)}{r }\right)^2\\ \\
&\sim& \dfrac{1}{(2r)^4} \displaystyle{\int_{B_{r/2}(x_0-re_n)}
 \dfrac{|\nabla h_+|^2}{|x-x_0|^{n-2}}dx \cdot
\int_{B_{r/2}(x_0+re_n)} \dfrac{|\nabla
h_-|^2}{|x-x_0|^{n-2}}dx} \\ \\
 &\leq& \displaystyle{\dfrac{1}{(2r)^4} \int_{B_{2r}(x_0)} \dfrac{|\nabla
h_+|^2}{|x-x_0|^{n-2}}dx\cdot \int_{B_{2r}(x_0)} \dfrac{|\nabla
h_-|^2}{|x-x_0|^{n-2}}dx} \\ \\
&=& \phi(2r) \leq \phi(1/3) \leq C_n.
 \end{array}
$$

\end{proof}

\section{Properties of solutions with star-shaped initial data}

\begin{lemma}\label{star-shaped}
If $\Omega_0$ and $u_0$ are star-shaped with respect to the ball
$B_{r_0}(0)\subset \Omega_0$, then $\Omega_t(u)$ and $u(\cdot,t)$
stays $\sigma$-close to star-shaped for all $0\leq t \leq
\frac{1}{3}\sigma^{1/5}$. (See Figure 2)
\end{lemma}

\begin{proof}
1. Observe that, for any $a>0$, the parabolic scaling
 $(x,t) \to (ax,a^2t)$ preserves both the heat operator and the boundary motion law in $(ST2).$
Therefore, for any $\sigma>0$ the function
$$
u_1(x,t):=u((1+\sigma)(x-x_0)+x_0, (1+\sigma)^2t)
$$
 is also a viscosity solution of (ST2) with corresponding initial data.

\vspace{5 pt}

2. Choose $x_0\in B_{r_0}(0)$. Take a small $c_0>0$ such that
$B_{r_0+c_0}(0) \subset \Omega_0$. We claim that for $ 0\leq \delta
\leq \sigma^{6/5}$,
\begin{equation}\label{claim202}
u_1(x, 0) \leq u(x,\delta) \hbox{ in } B_R(0)- B_{r_0+c_0}(0)
\end{equation}
if $\sigma$ is small enough.  To show \eqref{claim202}, let us
introduce another function
$$
\tilde{u}(x,0):= u((1+\frac{\sigma}{2})(x-x_0)+x_0,0).
$$
 Also let $v^-$ be the solution of (ST1) with initial data
$u_0^-$, and with $v^- =1$ on $\partial B_R(0)$. Then by comparison,
$-v^- \leq u$ and $B_R(0) -\Omega_t(-v^-) \subset \Omega_t(u)$.
Hence by Lemma~\ref{good} applied for $-v^-$,
$$
\Omega_0(\tilde{u})\subset\Omega_t(u)\hbox{ for } 0\leq t\leq
\sigma^{7/6}.
$$
Moreover, due to our assumption,
$$
\tilde{u}(x,0) \leq u_0(x).
$$
Therefore,  the maximum principle for caloric functions implies
$$w(x,t)\leq u(x,t)$$
 where $w$ solves the heat equation in the cylindrical
domain $D=\Omega_0(\tilde{u}) \times [0,\sigma^{7/6}]$ with initial
data $\tilde{u}(x,0)$ and zero boundary data on
$\partial\Omega_0(\tilde{u})\times [0,\sigma^{7/6}]$.

Now $w_t$ solves the heat equation in $D$,
$$
w_t=\Delta w \geq -C\hbox{ at }t=0,  \hbox{ and }w_t =0\hbox{ on }\partial\Omega_0(\tilde{u}).
$$
Therefore we conclude that $w_t \geq -C$ in $D$. In particular
\begin{equation}\label{est22}
w(x,\delta) \geq \tilde{u}(x,0) -C\delta.
\end{equation}
 Next we compare $u_1(x,0)$ with $w(x,\delta)$. Observe that for $x\in B_R(0)-B_{r_0+c_0}(0)$,
$$
\begin{array}{lll}
u_1(x,0) &=& \tilde{u}(x,0) + \int_{\sigma/2}^\sigma ((x-x_0)\cdot Du ((1+s)(x-x_0)+x_0, 0)) ds \\ \\
    &\leq& \tilde{u}(x,0)-c_0\sigma^{7/6} \\ \\
                 &\leq & \tilde{u}(x,0) -C\sigma^{6/5} \\ \\&\leq&
                 w(x,\delta) \leq u(x,\delta)
\end{array}
$$
for $ 0 \leq \delta \leq \sigma^{6/5}$, where the first inequality
follows from our assumption (I-d) on $u_0$, the second inequality
follows if $\sigma$ is sufficiently small, and the third inequality
follows from (\ref{est22}). Hence we conclude \eqref{claim202}.

\vspace{10pt}

3. Our goal is to prove that for $0\leq \delta \leq \sigma^{6/5}$,
\begin{equation}\label{step3}
u_1(x,t)\leq u_2(x,t):=u(x,t+\delta)
\end{equation}
 in
$(B_R(0)-B_{r_0+c_0}(0)) \times [0,\sigma^{1/5}].$  Note that the
inequality holds at $t=0$ by step 2. However, we needs a bit more
arguments since we do not know yet whether the lateral boundary data
on $\partial B_{r_0+c_0}(0)$ is properly ordered.

Suppose
$$
\Omega(u_1)\subset \Omega(u)\hbox{ for }0\leq t\leq t_0
$$
and $\Omega(u_1)$ contacts $\partial\Omega(u)$ for the first time at $t=t_0$.  Observe then that
$$
f(x,t):= u(x,t+\delta) - u_1(x,t)
$$ solves the heat equation in $\Omega(u_1)$ with nonnegative boundary data for
$0\leq t\leq t_0$, with $$f(x,0) \geq 0 \hbox{ in }
B_R(0)-B_{r_0+c_0}(0).$$ Indeed following the computation given
above, it follows that
$$
f(x,0) \geq c_0\sigma \hbox{ in }B_{r_0+c_0}(0)
-B_{r_0+\frac{c_0}{2}}(0).$$ On the other hand, due to the fact that
$w_t \geq -C$ and $\delta \leq \sigma^{6/5}$, we have
  $$
  f(x,0) \geq (w(x,\delta)-w(x,0))+(w(x,0)-u_1(x,0)) \geq -C\sigma^{6/5} \hbox{ in } B_{r_0+\frac{c_0}{2}}(0).
$$
Therefore we have $$f(x,t) > 0 \hbox{ on }
\partial B_{r_0+c_0}(0) \times [0, t_0]$$
 if $t_0<<1$. But then this contradicts Theorem~\ref{thm:cp}
applied to the region $(B_R(0)-B_{r_0+c_0}(0))\times [0,t_0]$.

\vspace{5 pt}

 4. From (\ref{step3}) of step 3,  we obtain
\begin{equation}\label{starshaped}
u((1+\sigma)(x-x_0)+x_0, (1+\sigma)^2t) \leq u(x,t+ \delta)
\end{equation}
 in $(B_R(0)-B_{r_0+c_0}(0))\times[0,\sigma^{1/5}]$
  for any  $x_0\in B_{r_0}(0)$, as long as $\sigma$ and $\delta$ are sufficiently small and
satisfy $0\leq \delta \leq \sigma^{6/5}$. As a result, for  $0\leq
t\leq \frac{1}{3}\sigma^{1/5}$, we can choose
$\delta=\sigma(2+\sigma)t \leq \sigma^{6/5}$ such that
$$
 (1+\sigma)^2t=t+\delta.
$$
It follows then from \eqref{starshaped} that the function
$u(\cdot,t)$ is $\sigma$-monotone with respect to the cone of
directions $W_x$ in $(B_R(0)-B_{r_0+c_0}(0))$ for $t\in
[0,\frac{1}{3}\sigma^{1/5}]$.
$$
( \hbox{ Here } \quad W_x = \{\nu\in S^n: \nu =\frac{x-x_0}{|x-x_0|} \hbox{ for some } x_0\in B_{r_0}(0)\}.)
$$
\end{proof}

\begin{remark}
 Observe that, due to $(I-b)$, we have for $x \in \Gamma_0$
\begin{equation}\label{est:time}
t(x,d) :=\min[\frac{d^2}{u^+(x-de_n,0)}, \frac{d^2}{u^-(x+de_n,0)}
]\in [ d^{7/6}, d^{5/6}] << d^{4/5}
\end{equation}
where $t(x,d)$ is the time it takes for the free boundary to
regularize in $B_d(0)$. Therefore,  $u(\cdot,t)$ is at least
$d^4$-monotone with respect to $W_x$ in $(B_R(0)-B_{r_0+c_0}(0))$
for $0\leq t\leq t(x_0,d)$.  This property will serve as a basis for
our regularization argument in section 3.
\end{remark}

\begin{lemma} {\bf (Harnack at $t=0$)} \label{lem:harnack0}
Let $x \in \Gamma_0$, then for all $s>0$ and for $0\leq t\leq
t(x,s)$ we have
$$
u^+(x-se_n,t) \leq C_1 u^+(x-se_n, 0)
$$
and
$$
u^-(x+se_n,t) \leq C_1 u^-(x+se_n,0)
$$
where $e_n=x/|x|$.
\end{lemma}

\begin{proof}
Let $v^+$ solve the one-phase Stefan problem (ST1) with initial data
$v^+_0(x) = u_0^+(x)$. Then $v^+$ is also a solution of (ST2) with
$u_0(x) \leq v^+_0(x)$, and thus by Theorem~\ref{thm:cp} we have
$$
u(x,t) \leq v^+(x,t).
$$
Therefore it follows from one-phase Harnack inequality applied for
$v^+(x,t)$  that
$$
u^+(x-se_n,t) \leq v^+(x-se_n,t) \leq C_1v^+(x-se_n,0)=C_1u(x-e_n,0)
$$
for $0\leq t\leq t_0$ where $t_0 = s^2/u(x-se_n,0) \geq t(x,s)$.

As for $u^-(x,t)$, we compare $u^-$ with the solution $v^-$ of (ST1)
with initial data $v^-_0(x)=u_0^-(x)$ and with boundary data
$v^-=1$ on $\partial B_R(0)$. The rest of the argument is parallel
to above.

\end{proof}

\begin{lemma} {\bf (Backward Harnack at $t=0$)} \label{backwardharnack0}
Let $x \in \Gamma_0$ and let $s>0$. Then for $  0\leq t\leq t(x,s)$
$$
u^{+}(x-se_n,0) \leq C_1 u^{+}(x-se_n, t)
$$
and
$$
u^{-}(x+se_n,0) \leq C_1 u^{-}(x+se_n, t)
$$
\end{lemma}

\begin{proof}
We will only show the lemma for $u^+$. The other part follows by a
parallel argument. Let $v^-$ solve (ST1) with initial data $u^-_0$
and with boundary data $1$ on $\partial B_R(0)$. Then $-v^-$ is also
a solution of (ST2) with $-v^-_0 \leq u_0$, and thus by
Theorem~\ref{thm:cp}, $-v^-\leq u$ and
$$
\{v^-=0\} \subset \{u \geq
0\}.
$$

Note that $\Omega(v^-)$ moves according to the  one-phase dynamics, which has been studied in detail by \cite{ck2}.
 In particular we know that $\Omega(v^-)$ will
be Lipschitz at each time. Moreover, for a boundary point $(x,t) \in
\Gamma(v^-)$ and $d:= {\rm dist}(x, \Gamma_0(v^-))$, the normal
velocity $V_{x,t}$  satisfies
\begin{equation}\label{speed}
V_{x,t} =|Dv^-(x,t)| \sim \frac{v^-(x+2de_n,0)}{2d}\leq d^{\beta-1}
 \leq t^{\frac{\beta-1}{2-\alpha}}
\end{equation}
where the last inequality follows from Lemma~\ref{good}. Let
$v^*(x,t)$ solve the heat equation in $\{v^-=0\}$ with initial data
$u_0(x)$ and boundary data $0$ on $\partial \{v^-=0\}$. Since
$$
\Omega(v^*)=\{v^-=0\}\subset \{u \geq0\},
$$
 we have $v^*(x,t) \leq u(x,t)$.
 Moreover, for any given $t>0$, $\tilde{v}^-(x,s):=v^-(\sqrt{t}x, ts)$
 satisfies the assumptions of Lemma~\ref{lem:almostharmonic}.
Thus it follows that $v^-(\cdot,t)$ is $t^a$-close to a
 harmonic function in $B_{\sqrt{t}}(x)$ for some $a>0$, where $x \in \Gamma_0$.
 Moreover, due to the assumption on the initial data,
 $(v^*)_t = \Delta v^* \geq -C$ at $t=0$. Also on $\Gamma(v^*)$,
  $$
   (v^*)_t /|D v^*|  =-(v^-)_t /|D v^-| =-|Dv^-| \geq
   -t^{\frac{\beta-1}{2-\alpha}}$$
where the last inequality follows from (\ref{speed}). Since
$\Omega(v^*)$ is Lipschitz and $ \Gamma_t(v^*)=\Gamma_t(v^-)$ is
regularized in time (Theorem~\ref{thm:ckmain}), (\ref{speed}) also
holds for $|Dv^*|$.

\vspace{10pt}

Hence on $\Gamma(v^*)$,
$$
(v^*)_t =-|Dv^-||Dv^*| \geq -t^{\frac{2(\beta-1)}{2-\alpha}} >
-t^{-2/5}.
  $$
Since $(v^*)_t$ solves a heat equation in $\Omega(v^*)$,  it follows
that for $x \in \Gamma_0$,
 \begin{equation} \label{v2}
 (v^*)_t \geq -t^{-2/5} \hbox{ in } B_{\sqrt{t}/2}(x-\sqrt{t}e_n)\times [0,t].
 \end{equation}
  Then since $v^*(x-\sqrt{t}e_n,0) \geq (\sqrt{t})^\alpha \geq (\sqrt{t})^{7/6}=t^{7/12}$,
\begin{eqnarray*}
u^+(x-\sqrt{t}e_n,0)=v^*(x-\sqrt{t}e_n,0) &\leq& C_1 v^*
(x-\sqrt{t}e_n,t) \\&\leq& C_1 u^+(x -\sqrt{t}e_n ,t)
\end{eqnarray*}
  where the first inequality follows from (\ref{v2}).
 Since $\Gamma(v^*)=\Gamma(v^-)$ is Lipschitz in a parabolic
scaling, $v^*$ is almost harmonic. Hence  $v^*(\cdot,t)$ is bigger
than the harmonic function $\omega^t(x)$ in $\Omega_t(v^*)\cap
B_{\sqrt{t}}(x)$ with its value
$$
\omega^{t}(x-\sqrt{t}e_n)=(C_1)^{-1}u^+(x-\sqrt{t}e_n,0).
$$  Note
that if $0 \leq t \leq t(x,s)$, then $s < \sqrt{t}$. Hence for $0
\leq t \leq t(x,s)$,
$$
C_1 u^+(x-se_n, t) \geq  C_1 v^*(x-se_n,t) \geq C_1\omega^t(x-se_n)
\geq Cu^+(x-se_n,0),
$$
where the last inequality follows  since the one-phase result
 implies a power law on the movement of
$\Gamma(v^-)=\Gamma(v^*)$ (see Lemma 2.5 of [CJK1]), and this yields
a bound on $u^+(x-se_n,0)/\omega^t(x-se_n)$.

\vspace{10pt}

Similar arguments apply to $u^-$, if we consider  the function $v^+$
solving (ST1) with initial data $u^+_0$, and the function
$v^{\star}$ solving the heat equation in $\{v^+=0\}$ with initial
data $u_0$ and with boundary data $0$ on $\Gamma(v^+)$ and $-1$ on
$\partial B_R(0)$.

\end{proof}

\begin{lemma} {\bf(Distance estimate at $t=0$)} \label{distance0}
Let $x \in \Gamma_0$ and let $s$ be a sufficiently small positive
constant. If
$$
\frac{|u^{+}(x - se_n,0)|}{s} \leq m \,\,\hbox{ and
}\,\,\frac{|u^{-}(x + se_n,0)|}{s} \leq m, $$ then for  $t \in [0,
\frac{s}{m}]$,
$$
d(x, t) = \sup\{r: x+re_n\hbox{ or } x-re_n\in\Gamma_t(u)\} \leq s.
$$
\end{lemma}
\begin{proof}
Let $v^+$ solve (ST1) with initial data $u_0^+$, and let $v^-$ solve
(ST1) with initial data $u_0^-$ and with $v^- =1$ on $\partial
B_R(0)$. Then by comparison, $ -v^- \leq u\leq v^+ $ and the lemma
follows from the one-phase result Theorem~\ref{thm:ckmain}.
\end{proof}

In the following lemma, we approximate our solution by harmonic
functions.

\begin{lemma}{\bf (Spatial regularity in the whole domain)}\label{regularity}
For $x_0\in\Gamma_0$ and $r>0$,  there exists a function
$\omega(x,t):=\omega^+(x,t) -\omega^-(x,t)$ such that
\begin{itemize}
\item[(a)] $\omega(\cdot,t)$ is harmonic in its
positive and negative phase in \\
$(1+r)\Omega_t(u)-(1-r)\Omega_t(u)$, and $\Omega(\omega^+)$,
$\Omega(\omega^-)$ are star-shaped;
\item[(b)] For a dimensional constant $C>0$, we have
$$
\omega^{+}(x,t)\leq u^{+}(x,t) \leq C\omega^{+}((1- r^{5/4})x,t)
$$
and
$$
\omega^{-}(x,t)\leq u^{-}(x,t) \leq C\omega^{-}((1 + r^{5/4})x,t)
$$
in  $B_r(x_0)\times [r^2,t(x_0,r)]$.
\end{itemize}
\end{lemma}

\begin{remark}
Note that we do not know yet whether the solution is close to a
Lipschitz graph in time. Also, note that  $t(x_0,r) \geq r^{7/6} \gg
r^2$, and  $\partial\{\omega^+>0\}$ need not be  $\partial
\{\omega^-
>0\}$.
\end{remark}

\begin{proof}

 1. We will only show the lemma for $u^+$. Let $\Gamma^\star$ be the
 free boundary  obtained from the one-phase problem (ST1) with the initial data
 $u^+_0$, and let $\Omega^*$ be the region bounded by $\Gamma^\star$. Let
 $v_1$ solve the heat equation in $\Omega^*$ and in $B_R(0) \times
 [0,1] -\Omega^*$, with initial data $u_0$ and with $v_1 =-1$ on $\partial
 B_R(0)$. Similarly, we define $v_2$, whose free boundary is obtained from the one-phase solution
 with initial data $u_0^-$. Then by comparison,
$$ v_2 \leq u \leq
v_1.$$ Hence the free boundary of $u$ is trapped between the free
boundaries of $v_1$ and $v_2$.  Also, since one-phase versions $v_1$
and $v_2$ behave nicely, we have those functions almost harmonic up
to $r$-neighborhood of their free boundaries for $r^2/2\leq  t \leq
r^2$. Next note that the range of $t$ is $0\leq t\leq t(x_0,r)$, and
thus both of the sets  $\Gamma_t(v_1)$ and $\Gamma_t(v_2)$ are
within distance $r$ of $\Gamma_0(u)$ in $B_r(x_0)$.  In particular,
using the one-phase result, i.e., arguing as in Lemmas  2.1 and 2.3
of [CK], we obtain
\begin{equation}\label{est0}
  v_2(x_0-2re_n,t) \sim u_0(x_0-2re_n,0) \sim v_1(x_0-2re_n,t)
\end{equation}
 for  $0\leq t\leq
  t(x_0,r)$.

\vspace{5pt}

 2. Observe that $$t(x_0, r) \leq r^{2-\alpha} \leq
r^{5/6}:=\tau.$$
 Due to Lemma~\ref{star-shaped}, we know that at
each time, $\Omega_t(u)$ is $\tau^5$
 -close to a star-shaped domain
$D_t$ up to the time $t=\tau$, i.e.,
\begin{equation}\label{est2}
D_t \subset \Omega_t(u) \subset (1+\tau^5)D_t  \subset (1+r^{4})D_t
\end{equation}
 for  $0 \leq t
\leq \tau$.
\begin{figure}
\center{\epsfig{file=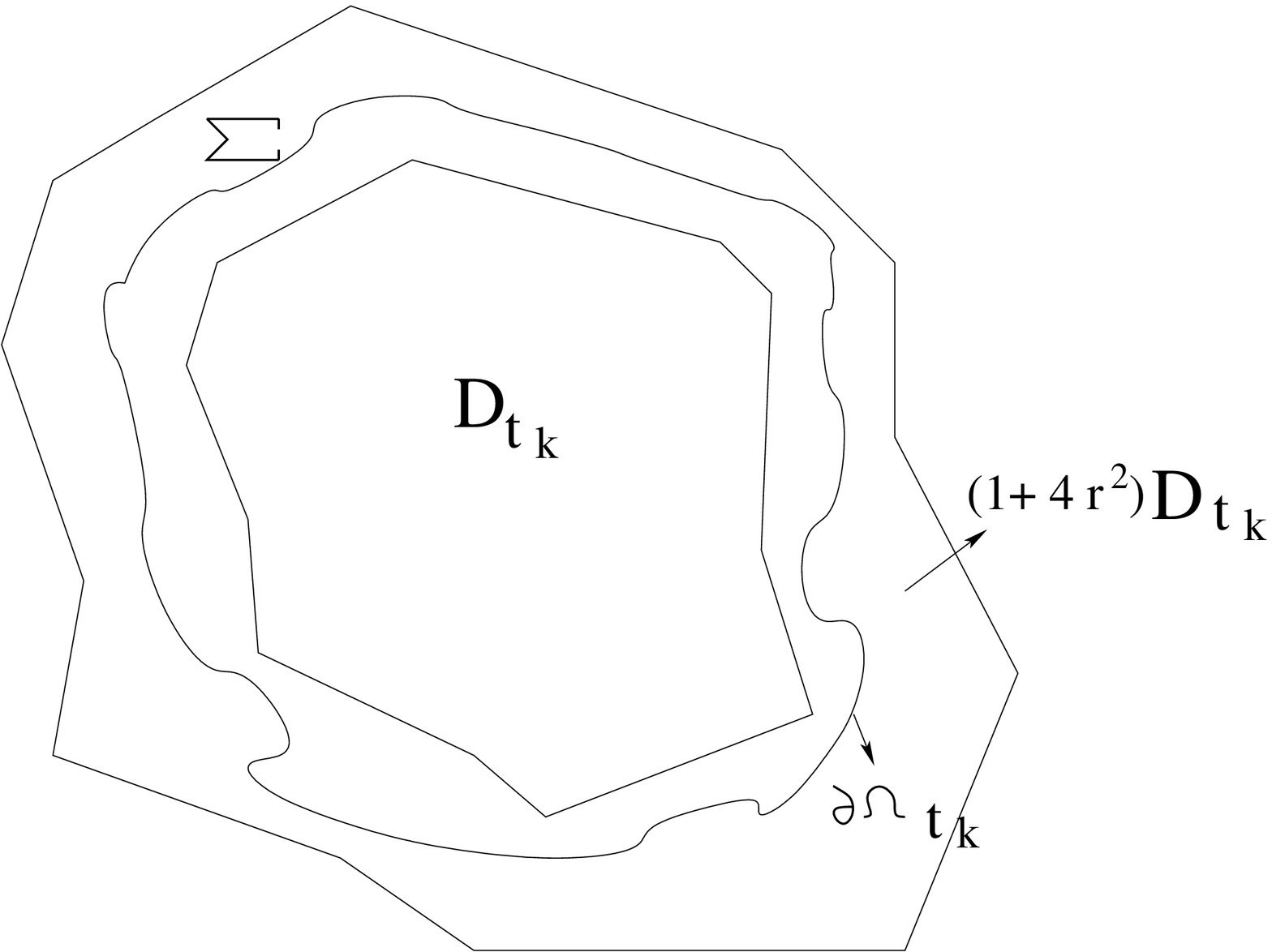,height=2.0in}} \center{Figure
2: Approximation of the positive phase by a star-shaped domain}
\end{figure}

   Then by Lemma~\ref{lem:harnack0} and (\ref{est2})
with $\beta \geq 5/6 $,
$$
u(x,t) \leq r^{(13/20)(5/6)}  =r^{13/24} \hbox{ on }
\partial (1-r^{13/20})D_0
$$
for $0 \leq t \leq \tau$. (Here note that we can apply
Lemma~\ref{lem:harnack0} up to the time $\tau$ since
$$
t(z,r^{13/20}) \geq r^{13(2-\beta)/20}  >\tau \hbox{ for any } z \in \Gamma_0.)
$$ Then by the $\tau^5$-monotonicity of $u$,
\begin{equation} \label{est3}
 u(x,t) \leq r^{13/24}  \hbox{ on } B_R(0) - (1-r^{13/20}+r^4)D_0
 \end{equation}
for $0 \leq t \leq \tau$.
 Since $\Gamma_t(u)$ stays in the
$\tau^{5/6}$-neighborhood of $\Gamma_0(u)$ up to  $\tau$, we obtain
that $\partial D_t$ stays in the $r^{25/36}$-neighborhood of
$\partial D_0$ up to the time $\tau$. Since $r^{25/36} <r^{13/20}$,
(\ref{est3}) implies
\begin{equation}\label{upperbound}
u(x,t) \leq r^{13/24}  \hbox{ on } B_R(0) - D_s
\end{equation}
 for any $0 \leq s,t
\leq \tau$.

 \vspace{5pt}

 3. Let
$$t_0=0 \leq t_1 =r^2 \leq t_2=2r^2 \leq ...\leq
t_{k_0}=k_0 r^2 \leq \tau$$ and fix a number $b$ such that
$$5/4 \leq b < 61/48.$$
We will construct a supersolution of (ST2) in
$$
(B_R(0)-(1+r^b)D_{t_k}) \times [t_k,t_{k+1}].
$$  Let $w^k(x)$ be the harmonic function in $$\Sigma:=(1+4r^b) D_{t_k} -
D_{t_k}$$ with boundary data zero on $\partial(1+4r^b) D_{t_k}$ and
$C_nr^{13/24}$ on $\partial  D_{t_k}$, where $C_n$ is a sufficiently
large dimensional constant. Extend $w(x)=0$ in $\R^n-\Sigma$.
 Next  define
$$
\Phi(x,t):= \inf\{\omega(y): |x-y| \leq
r^b-(t-t_k)\dfrac{r^{b-2}}{2}\}
$$
in $(B_R(0)-(1+r^b)D_{t_k}) \times [t_k,t_{k+1}]$. We claim that
$\omega$ is a supersolution of (ST2)  since our constant $b$
satisfies
\begin{equation}\label{claim101}
r^{b-2} >   r^{\frac{13}{24}-b}.
\end{equation}

To check this, first note that $\Phi(\cdot,t)$ is superharmonic in
its positive set and $\Phi_t \geq0$. Hence we only need to show that
\begin{equation}\label{momo}
\dfrac{\Phi_t}{|D\Phi|}  \geq |D\Phi| \hbox{ on }\Gamma(\Phi).
\end{equation}
Due to the definition of $\Phi$, $\Gamma_t(\Phi)$ has an interior
ball of radius at least $r^b/2$ for $t_k \leq t\leq t_{k+1}$. This
and the superharmonicity of $\Phi$ in the positive set yields that
$$
|D\Phi|  \leq \frac{C r^{13/24}}{r^b}  \hbox{ on } \Gamma(\Phi)
$$
for a dimensional constant  $C>0$.
 Moreover
$\Gamma(\Phi)$ evolves with normal velocity $ \frac{1}{2}r^{b-2}$.
Since \eqref{claim101} holds for our choice of $b$ (i.e., for $5/4
\leq b < 61/48$), we conclude \eqref{momo} for $r$ smaller than a
dimensional constant $r(n)$. Now we compare $u$ with $\Phi$ in
$$(B_R(0)-(1+r^b)D_{t_k}) \times [t_k,t_{k+1}].$$ Note that by
\eqref{upperbound},
$$u^+ \leq \Phi \hbox{ on }\partial (1+r^b)D_{t_k}$$
if $C_n$ is chosen sufficiently large.
 Also at $t=t_k$, (\ref{est2}) implies
 $$
u(\cdot, t_k) \leq 0  \leq \Phi(\cdot, t_k) \hbox{ on }
B_R(0)-(1+r^b)D_{t_k}.$$
  Hence we get $u \leq \Phi$ in $(\R^n-(1+r^b)D_{t_k})\times
[t_k,t_{k+1}]$.  This implies
\begin{equation}\label{order}
\Omega(u) \subset \Omega(\Phi) \cup ((1+r^b)D_{t_k}\times [t_k,
t_{k+1}]):=\tilde{\Omega}(\Phi)
\end{equation}
 for  $t_k \leq t \leq t_{k+1}$.

\vspace{5pt}

 4. Next we let $v(x,t)$ solve the heat equation in
$$
\tilde{\Omega}(\Phi)-((1-3r)\Omega_0(u)\times [t_k, t_{k+1}])
$$ with initial
data $v(\cdot, t_k)=u(\cdot,t_k)$ and boundary data zero  on
$\Gamma(\Phi)$ and $v=u$ on $(1-3r)\Gamma_0(u)$. Observe that, due
to \eqref{order}, we have $u^+ \leq v$ for $t_k \leq t\leq t_{k+1}$.

\vspace{10pt}

Since $\tilde{\Omega}(\Phi)$ is star-shaped and expands with
its normal velocity $< r^{b-2}$ which is less than $r^{-1}$,
Lemma~\ref{lem:almostharmonic} applies to $\tilde{v}(x,t):=
v(rx,r^2t)$. In particular there exists a constant $C>0$ such that
$$
(1/C)v(x,t) \leq  h_1(x,t) \leq Cv(x,t)
$$
for $(t_k+t_{k+1})/2 \leq t \leq t_{k+1}$, where $h_1(\cdot,t)$ is
the harmonic function in \\ $\Omega_t(v) - (1-2r)\Omega_0(u)$ with
boundary data zero on $\Gamma_t(v)$ and $v$ on $(1-2r)\Gamma_0(u)$.

Hence we conclude that
$$
u^+ \leq v \leq Ch_1 $$
 in $ (B_R(0)-(1-2r)\Omega_0(u)) \times [(t_k+t_{k+1})/2,t_{k+1}].
$

\vspace{5pt}

 5.  Similar arguments, now pushing the boundary purely by
the minus phase given by the harmonic function yields that
$$
B_R(0)- \tilde{\Omega}_t(\Psi):= \Pi_t \subset \Omega_t(u)
$$
for $t_k \leq t\leq t_{k+1}$, where
$$
\Pi_t = \{x\in D_{t_k}:  {\rm dist}(x,\partial D_{t_k}) \geq 3
r^b+\frac{r^{b-2}}{2}(t-t_k) \} .
$$
Let $w(x,t)$ solve the heat equation in
$$\Pi -  ((1-3r)\Omega_0(u) \times [t_k, t_{k+1}]))$$ with
initial data $u(\cdot,t_k)$ and boundary data zero on $\partial
\Pi$, and $u$ on $(1-3r)\Gamma_0(u)$. Then $u\geq w(x,t)$.

 Since $\Pi$ is star-shaped and shrinks with its normal velocity $<r^{b-2}$
which is less than $r^{-1}$, Lemma~\ref{lem:almostharmonic} applies
to $\tilde{w}(x,t):= w(rx,r^2t)$. In particular there exists $C>0$
such that
$$
u^+ \geq w \geq (1/C)h_2
$$
for $(t_k+t_{k+1})/2 \leq t \leq t_{k+1}$,
 where $h_2(\cdot,t)$ is the harmonic
function in \\
$\Pi_t -(1-2r)\Omega_0(u)$ with boundary data coinciding
with that of $w$.

\vspace{5pt}

 6. Lastly  we will show that $h_1$ and $h_2$ are not too
far away, i.e.
$$
h_1 (x,t)\leq Ch_2(x-8r^{b}e_n,t)
$$
with a dimensional constant $C>0$. Since $u$ is between $(1/C)h_2$
and $Ch_1$, this will conclude our lemma for $(t_k+t_{k+1})/2 \leq t
\leq t_{k+1}$. Then by changing the time intervals $[t_k, t_{k+1}]$
to $[t_{k}+r^2/2,  t_{k+1}+r^2/2]$, we obtain lemma for $r^2\leq r
\leq t(x_0,r)$.

 To prove the claim, observe
that
$$
\Omega_t(w) \subset \Omega_t(v) \subset (1+8r^b)\Omega_t(w)
$$
 Moreover, observe that
  \begin{eqnarray*}
  v_2((1+r^{b})x,(1+r^{b})^2 (t-t_k)+t_k)& \leq& v(x,t),w(x,t)\\
  &\leq& v_1((1-r^{b})x, (1-r^{b})^2(t-t_k)+t_k)
  \end{eqnarray*}
  for  $t_k \leq t\leq t_{k+1}$.
 This and  \eqref{est0} yield
 $$v(x_0-2re_n,t)\sim w(x_0-2re_n,t) \sim u(x_0-2re_n,0).$$
It follows that
$$
w(x,t) \leq v(x,t) \leq Cw(x-8r^{b}e_n,t)\hbox{ on } (1-2r)\Gamma_0
\times [t_k, t_{k+1}].
$$
Hence due to Dahlberg's lemma, we conclude that
$$
h_1(x,t)\leq C_1v(x,t) \leq C_2w(x-8r^{b}e_n,t) \leq
C_3h_2(x-8r^{b}e_n,t)
$$
in $B_r(x_0)\times [(t_k+t_{k+1})/2, t_{k+1}]$.  Since the
inequality holds for any $5/4 \leq b< 61/48$, we can conclude the
lemma.

\end{proof}

\begin{proposition} {\bf (Regularization in bad balls)} \label{pri-reg-bad}
For a fixed $x_0\in \Gamma_0(u)$,  suppose that either
$$
u^+(x_0-re_n,t_0) \geq M u^-(x_0+re_n,t_0)$$ or
$$
u^-(x_0+re_n,t_0) \geq M u^+(x_0-re_n,t_0)
$$
for $M>M_n$, where $M_n$ is a sufficiently large dimensional
constant. Then for $r \leq 1/M_n$, there exists a dimensional
constant $C>0$ such that
$$
|\nabla u^+(x,t)| \leq C\dfrac{u^+(x_0-re_n,t_0)}{r} \,\,\hbox{ and
}\,\, |\nabla u^-(x,t)| \leq C\dfrac{u^-(x_0+re_n,t_0)}{r}
 $$
 in
$B_r(x_0)\times [t(x_0,r)/2, t(x_0,r)].$

\end{proposition}

\begin{remark}
1. In the next section, we will extend this Lemma for later times,
i.e., for $x_0 \in \Gamma_{t_0}$. (See Lemma~\ref{reg-bad}.)\\

2. Note that the situation given in Proposition~\ref{pri-reg-bad} is essentially a perturbation of the one-phase case in \cite{ck}. The main step in the proof is in verification of this observation: i.e., by barrier arguments we will show that our solution is very close to a re-scaled version of the one-phase solution of (ST), for which the regularity of solutions are well-understood (see Theorem 1.1).

\end{remark}

\begin{proof}

 Without loss of generality, we may assume that
$$
u^+(x_0-re_n,0) \geq Mu^-(x_0+re_n,0).
$$

\vspace{5 pt}

1. First we show that after a small amount of time $u$ become almost
harmonic near the free bounadry. By  Lemmas ~\ref{lem:harnack0} and
~\ref{backwardharnack0} imply that for $0 \leq t\leq t(x_0,r)$,
\begin{equation} \label{simil}
u^+(x_0-re_n,t) \sim u^+(x_0-re_n,0)\hbox{,  } \quad u^-(x_0+re_n,t)
\sim u^-(x_0+re_n,0)
\end{equation}
 Also note that, by the assumption on the initial data $u_0$,
  Lemma~\ref{regularity} holds at $t=0$. In other words, there exists a
function $\omega(x,0)=\omega_0(x)$ such that
\begin{itemize}
\item[(a)] $\omega_0$ is harmonic
in its positive and negative phases in \\
$(1+r)\Omega_0(u)-(1-r)\Omega_0(u)$;
\item[(b)] $\Omega(\omega_0^+)$ and  $\Omega(\omega_0^-)$ are
star-shaped;
\item[(c)] In $B_r(x_0)$ , we have
\begin{equation} \label{prelim}
\omega_0^+(x)\leq u_0^+(x) \leq C\omega_0^+((1-r^{5/4})x)
\end{equation}
and
\begin{equation}\label{prelim:2}
\omega_0^-(x)\leq u_0^-(x) \leq C\omega_0^-((1+r^{5/4})x).
\end{equation}

\end{itemize}

\medskip
Next we improve (\ref{prelim}) and (\ref{prelim:2}) for later times,
and obtained the inequalities with $C=(1+r^a)$ for $t \geq r^{3/2}$.
 By the distance estimate-Lemma~\ref{good}, the free boundary of
 $u$ moves less that $r^{9/7}<r^{5/4}$ during  the time $t=r^{3/2}$.
  Then we let $v_1$ solve the heat equation in cylindrical domains
 $$
 (1+2r^{5/4})\Omega_0(\omega^+) \times [0, r^{3/2}]\cup(B_2(0)-(1+2r^{5/4})\Omega_0(\omega^+)) \times [0, r^{3/2}],
 $$ with initial data
 $u_0$ and lateral boundary data zero on $(1+2r^{5/4})\Gamma_0(\omega^+) \times
 [0,r^{3/2}]$, and $-1$ on $\partial B_2(0) \times
 [0,r^{3/2}]$.
Similarly, we let $v_2$
 solve  the heat equation in cylindrical domains
 $$
 (1-2r^{5/4})\Omega_0(\omega^+) \times [0, r^{3/2}] \cup (B_2(0)-(1-2r^{5/4})\Omega_0(\omega^+)) \times [0, r^{3/2}]
 $$ with initial data $u_0$ and lateral boundary data zero on $(1-2r^{5/4})\Gamma_0(\omega^+) \times
 [0,r^{3/2}]$,  and $-1$ on $\partial B_2(0) \times
 [0,r^{3/2}]$.
Then by comparison, $v_2 < u< v_1$. Also
 by Lemma~\ref{lem:almostharmonic} and $\beta\geq 5/6$,
 $$
 |v_1-v_2|\leq r^{\frac{5}{4}\times\frac{5}{6}} = r^{25/24}
 $$
 in the domain. Note that on $(1 - r^{6/7})\Gamma_0(\omega^+)$,
 $|v_1| \geq r^{\frac{6}{7}\times\frac{7}{6}}=r$ and thus
 $$
|v_1-v_2| \leq r^{a_1}|v_1|\hbox{ on } (1 - r^{6/7})\Gamma_0(\omega^+) \hbox{  for } a_1=1/24.
$$
 Similarly,
 $$
 |v_1-v_2| \leq r^{a_1}|v_2| \hbox{ on }  (1 + r^{6/7})\Gamma_0(\omega^+).
 $$
 Then since $v_1$ and $v_2$ are almost harmonic in the $r^{3/4}$-neighborhood of their boundaries
 for $\frac{1}{2}r^{3/2} \leq t\leq r^{3/2}$,
 the above inequalities
 on $|v_1-v_2|$ imply the following:
 for $\frac{1}{2}r^{3/2} \leq t\leq r^{3/2}$, there exist positive harmonic
 functions $\tilde{\omega}^+(\cdot, t)$ and $\tilde{\omega}^-(\cdot, t)$
 defined
respectively in $$\Omega_t(v_2^+) \cap (B_R(0) -
(1-r^{1-b})\Omega_0(\omega^+)) \hbox{ and }\Omega_t(v_1^-) \cap
(1+r^{1-b})\Omega_0(\omega^+))$$
 where $b=1/7$, such that for some $a>0$
\begin{equation} \label{eq2}
\tilde{\omega}^+(x,t)\leq u^+(x,t) \leq
(1+r^{a})\tilde{\omega}^+((1-4r^{5/4})x,t)
\end{equation}
and
\begin{equation}\label{eq2:1}
\tilde{\omega}^-(x,t)\leq u^-(x,t) \leq
(1+r^{a})\tilde{\omega}^-((1+4r^{5/4})x,t).
\end{equation}

Now on the time interval $[0, r^{3/2}]+ \frac{k}{2}r^{3/2}$, $1 \leq
k \leq m$, we construct $v_1$ and $v_2$ so that they solve the heat
equation in the  cylindrical domains with
$$\Gamma(v_1) = (1+2r^{5/4})\Gamma_{\frac{k}{2}r^{3/2}}(\omega^+)
\times [\frac{k}{2}r^{3/2}, (1+\frac{k}{2})r^{3/2}]$$ and
$$\Gamma(v_2) = (1-2r^{5/4})\Gamma_{\frac{k}{2}r^{3/2}}(\omega^+)
\times [\frac{k}{2}r^{3/2}, (1+\frac{k}{2})r^{3/2}].$$ Then by a
similar argument as above,  we obtain harmonic functions
$\tilde{\omega}^\pm(\cdot, t)$ satisfying (\ref{eq2}) and
(\ref{eq2:1}) for
$$
\frac{1+k}{2}r^{3/2} \leq t\leq (1+\frac{k}{2})r^{3/2}.
$$ Hence we
conclude (\ref{eq2}) and (\ref{eq2:1}) for $r^{3/2} \leq t\leq
t(x_0,r)$.

\vspace{5 pt}

 2. Next we re-scale $u(x,t)$ as follows:
$$
\tilde{u}(x,t):= \alpha^{-1} u(rx+x_0,r^2\alpha^{-1} t)\hbox{ in }
2Q_{x_0},
$$
where $\alpha:=u^+(x_0-re_n,t_0) < 1.$ Then $\tilde{u}(x,t)$ solves
$$
\left\{\begin{array}{lll}
(\alpha\partial_t -\Delta)\tilde{u}=0&\hbox{ in }& \Omega(\tilde{u})\\ \\
V=|D\tilde{u}^+|-|D\tilde{u}^-| &\hbox{ on } &\Gamma(\tilde{u})\\
\\
\tilde{u}(-e_n,0)=1 \\ \\
\tilde{u}(e_n,0) =  -1/N  &\hbox{ where }&  N \geq M.
\end{array}\right.
$$
Furthermore, (\ref{simil}) implies that for $0\leq t\leq 1$,
$$
\tilde{u}^+(-e_n,t)\sim 1 \hbox {,  }\quad \tilde{u}^-(e_n,t)\sim
\frac{1}{N}.
$$
Let $\tilde{w}$ be the corresponding re-scaled version of
$\tilde{\omega}$ given in \eqref{eq2} and \eqref{eq2:1}, then in
$B_{r^{-b}}(0)\cap \Omega_0(\tilde{u})$ we have
\begin{equation}\label{improved1}
(1-r^a) \tilde{w}^+((1+4r^{5/4})x, \alpha r^{-1/2})\leq
\tilde{u}^+(x, \alpha r^{-1/2}) \leq \tilde{w}^+(x, \alpha r^{-1/2})
\end{equation}
and
\begin{equation}\label{improved2}
(1-r^a)\tilde{w}^-(x, \alpha r^{-1/2})\leq \tilde{u}^-(x, \alpha
r^{-1/2}) \leq \tilde{w}^-((1+4r^{5/4})x,  \alpha r^{-1/2})
\end{equation}
Here note that
$$ \alpha r^{-1/2} = \sqrt{r} \cdot \frac{u^+(x_0-re_n,t_0)}{r} \leq r^{1/3}.$$
Lastly,  for given $x_0\in\Gamma(\tilde{u})\cap B_1(0)$, a similar
argument as in  (\ref{v2}) implies that
 \begin{equation}\label{overtime}
 \tilde{u}(x,t) \leq (1+r^b)\tilde{u}(x,0)\hbox{ in }\partial B_{\frac{1}{2}r^{-b}}(r^{-b}e_n)\times [0,1].
 \end{equation}

\vspace{5 pt}

3. We claim that we can construct a supersolution $U_1$ and a
subsolution $U_2$ of (ST2)  such that
$$
U_2(x,t) \leq \tilde{u}(x,t) \leq U_1(x,t) \leq  U_2(x- \sqrt{\e}
e_n,t)\hbox{ in } B_1(0)\times [\alpha r^{-1/2},1]
$$
and that $U_2$ is a smooth solution with uniformly Lipschitz
boundary in space and time. Then for sufficiently small $r>0$ the
lemma will follow from analysis parallel to that of \cite{acs2}.

To illustrate the main ideas, let us first assume that
\begin{itemize}
\item[(a)] \eqref{improved1} and \eqref{improved2}
hold in the entire ring domain $R\times [0,1]$,
 where
 $$R=\{x:d(x,\Gamma_0(\tilde{u})) \leq r^{-b}\};$$
\item[(b)] $\tilde{u}(x,t) \leq (1+r^b)\tilde{u}(x,0)
\hbox{ on } \partial R\times [0,1]$.
\end{itemize}

Let $U^+_1$ be the solution of (HS) in
$\Sigma=(\R^n-(\Omega_0-R))\times [0,1]$ with initial data
$\tilde{w}(x,t)$ and boundary data $(1+r^b)\tilde{u}(x,0)$, and let
$$
U_1=U_1^+-U_1^-\hbox{ in } R\times [0,1],
$$ where $U_1^-(\cdot,t)$ is the harmonic function in $R-\Omega(U_1^+)$
with fixed boundary data zero on $\Gamma(U_1^+)$ and $C/N$ on
$\partial R - \Omega(U_1^+)$.  Then $U_1$ is a supersolution of
(ST2) in $\Sigma$, and thus by Theorem~\ref{thm:cp} and the
 assumptions (a)-(b) we have $ \tilde{u}\leq U_1$ in $\Sigma$.

\vspace{5pt}

4. The construction of the subsolution $U_2$ is a bit less
straightforward. We use
$$
U^+_2(x,t):= (1-\e)\sup_{|y-x|\leq
\sqrt{\e}(1-c(t))}U_1^+((1+\sqrt{\e})y,t),
$$
where $\e = 1/N$ and $c(t) := t^{4/5}$. Then we define
$$
U_2 = U_2^+-U_2^- \hbox{ in } R\times [0,1],
$$
 where  $R$ is the ring domain as given above and
$U_2^-(\cdot,t)$ is the harmonic function in $R-\Omega(U_2^+)$ with
fixed boundary data zero on $\Gamma(U_2^+)$ and $C/N$ on $\partial R
- \Omega(U_2^+)$. Then $U_2$ satisfies the free boundary condition
$$
V_{U_2} \leq (1+\e)|DU_2^+| -\sqrt{\e}c'(t).
$$
Therefore, $U_2$ is a subsolution of (ST2) if we can show  that
\begin{equation}\label{question}
\sqrt{\e}c'(t) \geq \e|DU_2^+| + |DU_2^-| \hbox{ on } \Gamma(U_2)
\end{equation}
and  $\int_0^1 c'(s) ds\leq 1$.

\vspace{5pt}

The analysis performed in \cite{ck}, as in the proof of (c) of
Theorem~\ref{thm:ckmain}, yields the following: at a fixed time $t$,
$\Gamma(U_1)$ regularizes in the scale of $d:=d(t)$ which solves
$$
t=\frac{d^2}{U_1(-de_n,0)}.
$$
Therefore,
$$
|DU_2^+| \sim \frac{U_2^+(-de_n,0)}{d}\hbox{ and } |DU_2^-| \sim
\frac{U_2^-(de_n,0)}{d}
$$
on
$$
\Gamma(U_2) \times [t/2,t].
$$
Observe that since $\beta \geq 5/6$,
 $$
 U_2^+(-de_n,0)  \leq d^{5/6}\hbox{ and } U_2^-(de_n,0) \leq \e d^{5/6},
 $$ then we have
$$
\e\frac{U_2^+ (-de_n,0)}{d}+ \frac{U_2^-(de_n,0)}{d}\leq \e
d^{-1/6} \leq \sqrt{\e} t^{-1/5}.
$$
where the last inequality follows from
$$
t=d^2/U_1(-de_n,0) \leq d^2/d^\alpha \leq d^{5/6}.
$$
 Hence $c(t) = t^{4/5}$ satisfies \eqref{question},
and we conclude that $U_2$ is a subsolution of (ST2).

 Now we can use the fact
$$
U_2\leq \tilde{u}\leq U_1\hbox{ in } B_c(0)\times [0,c]
$$
 to conclude that $\tilde{u}$ is $\sqrt{\e}$- close to  $U_1$: a Lipschitz (and smooth) solution
in $B_1(0)\times [1/2,1]$. Once we can confirm this, everything else
follows from analysis parallel to that of \cite{acs2} with the
choice of a sufficiently small $\e$.

\vspace{5pt}

5.   Now we proceed to the general proof without the simplified
assumptions (a) and (b) in step 3, which  are replaced with local
inequalities \eqref{improved1}-\eqref{improved2} and
\eqref{overtime}. For this we need to perturb the initial data
outside of $B_1(0)$ (see section 4, p 2781-2783 of \cite{cjk2}), to
obtain  functions $W_1(x)$ and $W_2(x)$  which satisfies the
followings:
\begin{itemize}
\item[(a)]  $\{W_k>0\}$ with $k=1,2$ is star-shaped and coincides with
$\Omega_{\alpha r^{-1/2}}(\tilde{w})$ in $B_{r^{-b}}(0)$;
\item[(b)] $\{W_2>0\}\subset\Omega_{\alpha r^{-1/2}}(\tilde{w})\subset \{W_1>0\}$ ;
\item[(c)] $d(x, \{W_k>0\} ) \geq r^{-b}$ with $k=1,2$ for
$x\in\Gamma_{\alpha r^{-1/2}}(\tilde{w}) \cap (\R^n-
B_{2r^{-b}}(0))$;
\item[(d)] $W_k$ is harmonic in $\{W_k>0\}-K$ with
boundary data zero on $\Gamma(W_k)$ and $(1+r^b)\tilde{w}(x,\alpha r^{-1/2})$ on $\partial K$,
where
$$
K = \{x: d(x,\Gamma(W_k)) \geq r^{-b}\}.
$$
\end{itemize}

Let $U_k$ be the solution of Hele-Shaw problem in
$$
\R^n-\frac{1}{2}\{W_k>0\}\times [\alpha r^{-1/2},1]
$$ with initial
data $W_1$ and with lateral boundary data $(1+r^b)\tilde{w}(x,\alpha
r^{-1/2})$. Due to Proposition 4.1 of \cite{cjk2}, for sufficiently
small $r>0$,
 the level sets of $U_1$ is then $\e c$-close to those of $U_2$ in $B_1(0)\times [0,1]$.
Hence we can use $U_2$ instead of $U_1$ in step 4. and proceed as in
step 4 to conclude.
 \end{proof}

\section{Decomposition based on local phase dynamics}

Throughout the rest of the paper, we fix $x_0\in\Gamma_0$ and  a sufficiently
small constant $r>0$, and will prove the regularization of the solution in
$B_{r}(x_0) \times [t(x_0,r)/2, t(x_0,r)]$. We also fix a constant $M \geq M_n$,
where $M_n$ is  a  sufficiently large dimensional constant.  If the ratio
between $u^+(x_0-re_n,0)$ and $u^-(x_0+re_n,0)$  is bigger than $M$,
 then we can directly apply
Proposition~\ref{pri-reg-bad} to prove the main theorem. Therefore
we  assume that
\begin{equation}\label{est101}
M^{-1} u^-(x_0+re_n,0)\leq u^+(x_0-re_n,0) \leq M u^-(x_0+re_n,0).
\end{equation}
Let
$$
C_0:= \max[\dfrac{u^+(x_0-re_n,0)}{r}, \dfrac{u^-(x_0+re_n,0)}{r}].
$$
Then since $u_0^+$ and $u_0^-$ are comparable with harmonic functions,
 $C_0$ is less
than a constant depending on $n$ and $M$ (See Corollary~\ref{moncor}). Also note that
$$
C_0 \geq r^{\alpha-1} \geq r^{1/6}.
$$
 Let
$$
A^+= \{x \in \Gamma_0\cap B_{2r}(x_0): \dfrac{u^+(x-se_n,0)}{s} \geq
MC_0 \hbox{ for some } r^{5/4}\leq s\leq r\}
$$
and
$$
A^-=\{x\in\Gamma_0\cap B_{2r}(x_0): \dfrac{u^-(x+se_n,0)}{s} \geq MC_0
\hbox{ for some } r^{5/4} \leq s\leq r\}.
$$
Denote  $$ A=A^+ \cup A^-.$$

\begin{lemma}\label{bigbigno}
If $$\dfrac{u^\pm(x\mp
se_n,0)}{s} \geq MC_0 \hbox{ for some } s \leq r,$$ then
$$\dfrac{u^\mp(x\pm se_n,0)}{s}\leq C_0.$$
\end{lemma}
\begin{proof}
Since $u^\pm_0$ are comparable with harmonic functions $h^{\pm}$, we can argue similarly as in
 Corollary~\ref{moncor}. Observe
 $$\begin{array}{lll}
 \dfrac{u_0^+(x- se_n)}{s} \cdot \dfrac{u_0^-(x+ se_n)}{s} &\sim& \dfrac{h^+(x- se_n)}{s} \cdot \dfrac{h^-(x+ se_n)}{s}\\ \\

 &\lesssim& \sqrt{\phi(r)} \lesssim C_0^2.
 \end{array}
 $$
\end{proof}

Now for $x
\in A^+ $,  we can find the largest constant $r_x < r$ such that
$$
 \dfrac{u^+(x-r_xe_n,0)}{r_x} = MC_0
 $$
then  let
$$
Q_x = B_{r_x}(x) \times [0, \frac{r_x}{MC_0}].
$$
Also for $x \in A^- $, we can similarly define $r_x$ and $Q_x$.
 Let
\begin{equation}\label{def01}
\Sigma :=B_{r}(x_0)\times [0,t(x_0,r)] - \bigcup_{x \in A}
Q_x.
\end{equation}
(See Figure 3)
\begin{figure}
\center{\epsfig{file=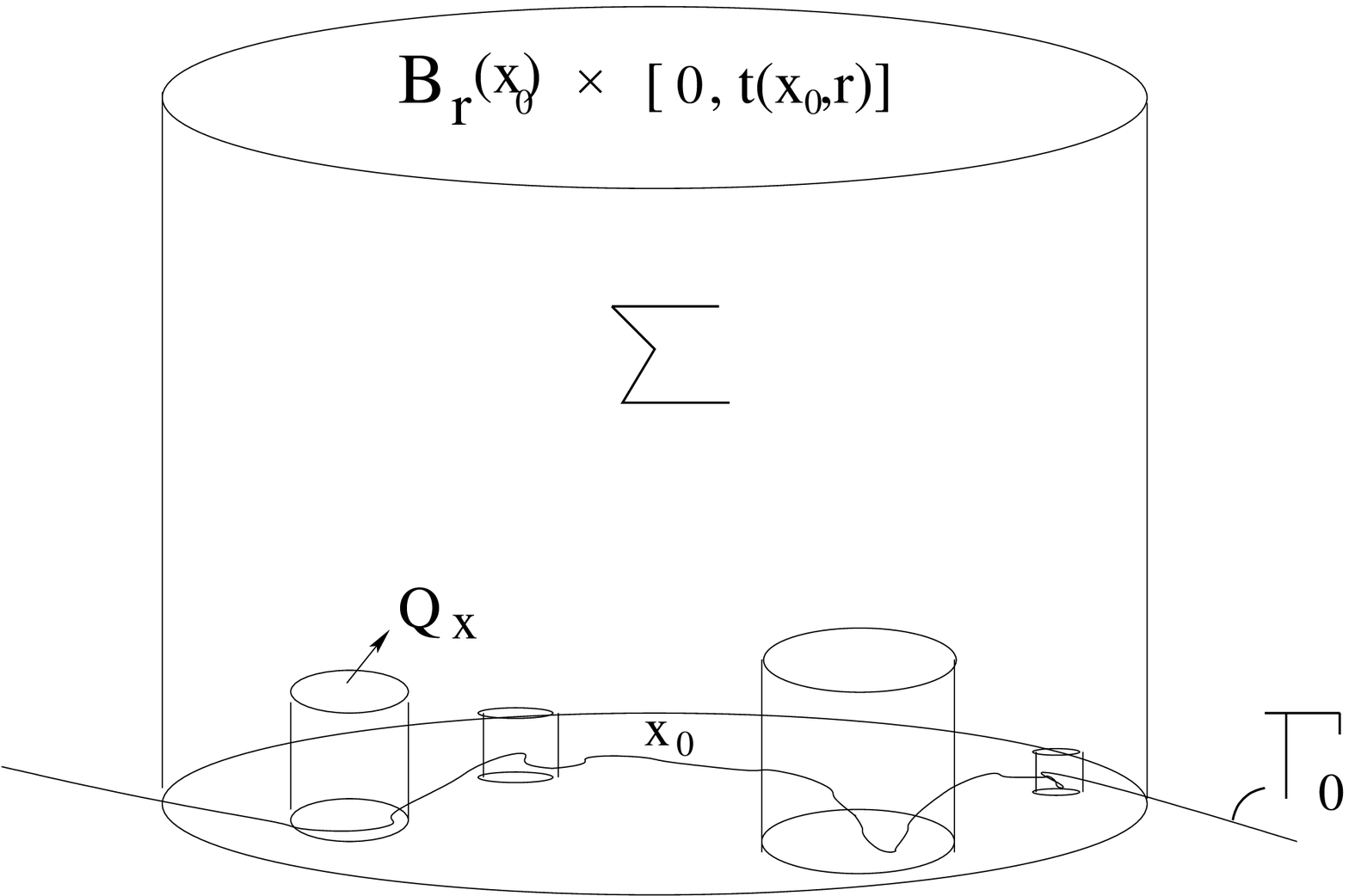,height=2.0in}} \center{Figure 3:
Decomposition of the domain}
\end{figure}

The following statement is a direct consequence of the definition  \eqref{def01}.
\begin{lemma} \label{measure}
If $x \in \Gamma_0 \cap \Sigma_0$, then for all $r^{5/4}\leq s \leq
r$
$$
\dfrac{u^+(x-se_n,0)}{s}, \dfrac{u^-(x+se_n,0)}{s}\leq MC_0.
$$
\end{lemma}

The next proposition is the main result in this section, which states
that the solution is ``well-behaved" in $\Sigma$.

\begin{proposition}\label{main1}
There exists a dimensional constant $K>0$ such that for all
$(x,t)\in\Gamma\cap\Sigma$
$$
\dfrac{u^{+}(x-se_n,t)}{s},\dfrac{u^{-}(x+se_n,t)}{s} <KMC_0 \,\,
\hbox{ for
 }\,\, r^{5/4} \leq s\leq r.\leqno(A)
$$
\end{proposition}

Before proving Proposition~\ref{main1}, we show an immediate consequence of
the proposition:
 we are ready to show that $\Gamma(u)$ is  close to a Lipschitz graph in time as well as in space.

\begin{corollary} \label{cor}
for $(x,t)\in\Gamma \cap \Sigma$, suppose
$(x+ke_n,t+\tau)\in\Gamma$. Then
$$
|k| \leq  r^{5/4} \,\,\hbox{ if }\,\, \tau \in [0,
\frac{r^{5/4}}{K_1MC_0}].
$$
where  $K_1$ is a dimensional constant.
\end{corollary}
\begin{proof}
 Due to Lemma~\ref{regularity},
at any time $0\leq t\leq t(x_0,r)$, we have
\begin{equation}\label{reminder}
h^{\pm}(x,t) \leq  u^{\pm}(x,t) \leq C_1h^{\pm} (x\mp r^{5/4}e_n,t).
\end{equation}
 in  $B_r(x_0)$, where $h:=h^{+}(\cdot, t)-h^-(\cdot,t)$ is harmonic
 in its positive and negative phase in $(1+r)\Omega_t(u) -(1-r)\Omega_t(u)$,
  and the domains $\Omega(h^{+})$ and $\Omega(h^{-})$ are both star-shaped with respect to $B_{r_0}(0)$.
 Let us pick $(y_0,t_0)\in \Gamma\cap\Sigma$. Due to
Proposition~\ref{main1}, \eqref{reminder} and the Harnack inequality
for harmonic functions,  we have
\begin{equation}\label{bound1}
 \sup_{y\in B_{10r^{5/4}}(y_0)} u(y,t_0) \leq CC_1KMC_0r^{5/4}
\end{equation}
where $C$ is a dimensional constant.
On the other hand, due to Lemma~\ref{star-shaped} and $t_0^5\leq r^{25/6}$, we have
\begin{equation}\label{bound2}
u(\cdot,t_0) \leq 0 \hbox{ in } B_{\frac{1}{2}r^{5/4}}(y_0+r^{5/4}e_n) .
\end{equation}

Let
$$
y_1:= y_0+r^{5/4}e_n,\,\, C_2 := CC_1KMC_0,\,\,r(t) :=
\frac{1}{2}r^{5/4}- C_3(t-t_0)
$$
where $C_3 = CC_2$.
 Next we define $\phi(x,t)$ in the domain
$$
\Pi:=B_{2r^{5/4}}(y_1) \times [t_0, t_0+ \dfrac{r^{5/4}}{C_3}]
$$
such that
$$
\left\{\begin{array}{lll}
-\Delta\phi(\cdot,t) = 0 &\hbox{ in } & B_{2r^{5/4}}(y_1) -B_{r(t)}(y_1) \\ \\
\phi = 2C_2 r^{5/4} & \hbox{ on  } &  \partial B_{2r^{5/4}}(y_1)\\ \\
\phi=0 &\hbox{ in }&  B_{r(t)}(y_1).
\end{array}\right.
$$
Then by (\ref{reminder}), (\ref{bound1}) and (\ref{bound2}), $u
\prec \phi$ at $t=t_0$ in $\Pi$. Let $T_0$ be the first time where
$u$ hits $\phi$ from below in $\Pi$. Since \eqref{bound1} also holds
for any $(x,t) \in \Gamma \cap \Sigma$ in place of $(y_0, t_0)$, we
have $u< \phi$ on the parabolic boundary of $\Pi\cap \{t_0 \leq
t\leq T_0\}$.  On the other hand, if $C$ is chosen sufficiently
large, then
 $$
 \dfrac{\phi_t}{|D\phi|} =C_3 \geq  |D\phi|
 \hbox{ on } \partial B_{r(t)}(y_1)\times [t_0,t_1:=t_0+ \dfrac{r^{5/4}}{4C_3}],
 $$
 and thus $\phi$ is a supersolution of (ST). This and Theorem~\ref{thm:cp}
 applied to $u$ and $\phi$ in $\Pi$ yields a contradiction,
 and we conclude that $\Gamma(u)$ lies outside of $B_{\frac{1}{4}r^{5/4}}(y_0+r^{5/4}e_n)$
 for $t_0 \leq t\leq t_1$.
Similarly, by constructing a negative radial barrier
 and comparing it with $u$, one can show that $\Gamma(u)$
 lies outside of $B_{\frac{1}{4}r^{5/4}}(y_0-r^{5/4}e_n)$ for $t_0 \leq t\leq t_1$. Hence we conclude.

\end{proof}
We proceed to show our main result, Proposition~\ref{main1}.
The following lemmas are used in the proof of the proposition.
\vspace{5 pt}

\noindent $\bullet$ For $x_0 \in \Gamma_{t_0}$, define
$$
 t(x_0,r) := \min [ \frac{r^2}{u^+(x_0-re_n,t_0)},  \frac{r^2}{u^-(x_0+re_n,t_0)}].
 $$

\vspace{5pt}

\begin{lemma} [{\bf Harnack at later times}] \label{laterharnack}
Fix $s \in [r^{5/4}, r] $. If $(y_0, t_0)  \in \Gamma \cap \Sigma $, then
$$
u^{+}(y_0-se_n,t_0) \geq c_1u^{+}(y_0-se_n, t_0+ \tau)
$$
and
$$
u^{-}(y_0+se_n,t_0) \geq c_1u^{-}(y_0+se_n, t_0+ \tau)
$$
 for  $0\leq \tau \leq  t(y_0,s)/2$ and $c_1>0$.
\end{lemma}

\begin{proof}

 We will show the lemma for $u^+$:  the statement on $u^-$ follows
via parallel arguments.

\vspace{12 pt}

1. Let $(y_0, t_0) \in \Gamma  \cap \Sigma$ and let $s \in [r^{5/4}, r] $. Let $h^+$ be  given as
in (\ref{reminder}). Due to Lemma~\ref{lem:harnack0} and
Lemma~\ref{backwardharnack0}, we have
\begin{eqnarray*}
h^+(y_0-2re_n,t_1) &\leq& u^+(y_0-2re_n,t_1) \\ &\leq& Cu^+(y_0-2re_n,t_2)
\leq Ch^+(y_0-(2r+r^{5/4})e_n,t_2)
\end{eqnarray*}
for $0\leq t_1,t_2 \leq t_0+t(y_0, r)/2$. (Here note that $y_0 \in
B_{r}(x_0)$.) In particular
\begin{equation}\label{insideharnack}
u^+(y_0-2re_n,t)\leq Ch^+(y_0-(2r+r^{5/4})e_n,t_0) \leq
C_1h^+(y_0-2re_n,t_0)
\end{equation}
for $t \leq t_0+t(y_0, s)/2 $.

\vspace{12 pt}

2. Now let $v^+$ solve (ST1) in $(\R^n-(1-2r)D_{t_0}) \times [t_0,
t_0+ t(y_0,s)/2]$ with initial and boundary data $C_2h^+(x-2se_n ,t)$.
Since $s \geq r^{5/4}$, (\ref{reminder}) implies
\begin{equation} \label{ob}
\Omega_t(u) \subset \Omega_{t_0}(v^+)\subset \Omega_t(v^+) \hbox{ in
}  B_{2s}(y_0)\times [t_0,t_0+t(y_0,s)/2].
\end{equation}
Then by (\ref{ob}), \eqref{insideharnack} and (\ref{reminder}),
$$u^+ \leq v^+
\hbox{ in }B_s(y_0)\times [t_0,t_0+t(y_0,s)/2]$$ if we choose $C_2$ as a
multiple of $C_1$ by a dimensional constant.
Moreover, due to the Harnack inequality for one-phase (ST1), one can
conclude that
$$
\begin{array}{lll}
u^+(y_0-se_n,t_0+\tau) &\leq&  v^+(y_0-se_n,t_0+\tau)\\ \\
& \leq& Cv^+(y_0-se_n,t_0) \\ \\
 &=&CC_2h^+(y_0-3se_n ,t_0)  \\ \\
 &\leq& C_3 h^+(y_0-se_n,t_0) \\ \\
&\leq&  C_3 u^+(y_0-se_n,t_0)
\end{array}
$$
 for  $0\leq \tau\leq \dfrac{s^2}{v^+(y_0-se_n,t_0)} \sim
t(y_0,s)/2$. Here the first inequality uses $u^+\leq v^+$, the second
uses the Harnack inequality for $v^+$, the third one uses the
Harnack inequality for harmonic functions and the last one uses
(\ref{reminder}).
\end{proof}

\begin{lemma} [{\bf Backward harnack}]\label{backward}
 Suppose that (A) holds up to time $t=T_0\leq
t(x_0,r)$. If $(y_0,t_0)\in\Gamma$ and $t_0 \leq T_0$, then for  $0
\leq \tau \leq t(y_0, s)/2$,
$$
u^{+}(y_0-se_n,t_0) \leq C u^{+}(y_0-se_n,t_0+\tau)
$$
and
$$
u^{-}(y_0+se_n,t_0) \leq C u^{-}(y_0+se_n,t_0+\tau)
$$
where $0 \leq s \leq r$ and $C$ is a universal constant.
\end{lemma}

\begin{proof}
We will  show the argument for $u^+$, due to the symmetric nature of
the claim. The argument here will be similar to that of
Lemma~\ref{backwardharnack0}, replacing the initial data $u_0^+$ and
$u_0^-$ (used in the construction of barriers) by $h^+(x,t_0)$ and
$h^-(x,t_0)$ given in (\ref{reminder}).

We consider $v_1$: a one-phase solution of (ST1) in
$$\Pi:=(1+r)\Omega_{t_0}\times [t_0, t_0+t(y_0,s)/2] $$ with
initial and lateral boundary data $C_1h^-$. Then $v_1\leq u$ in $\Pi$.
Now let $v_2$ solve the heat equation in $\{v_1=0\}\times
[t_0,t_0+t(y_0,s)/2]$ with initial data
 $$
 v_2(\cdot,t_0)=
\left\{ \begin{array}{lll}
 h^+(\cdot,t_0)&\hbox{ in }&\{v_1(\cdot,t_0)=0\} - (1-r)\{h^+(\cdot,t_0)>0\}\\ \\
 \tilde{h}(\cdot) &\hbox{ in } & (1-r)\{h^+(\cdot,t_0)>0\},
 \end{array}\right.
 $$
 where $\tilde{h}(\cdot)$ is a $C^2$ extension function of $ h^+(\cdot,t_0)$
 chosen so that $\tilde{h}(\cdot) \leq u^+(\cdot,t_0)$.
The rest of the proof is the same as that of
Lemma~\ref{backwardharnack0}.

\end{proof}

\begin{lemma} {\bf (Regularization in bad balls)} \label{reg-bad}
For a fixed $(x_0,t_0)\in \Gamma(u)$, and suppose
$$
u^+(x_0-re_n,t_0) \geq M u^-(x_0+re_n,t_0)$$ or
$$u^-(x_0+re_n,t_0) \geq M u^+(x_0-re_n,t_0)$$
for $M>M_n$, where $M_n$ is a dimensional constant. Then for $r \leq
1/M_n$, there exists a dimensional constant $C>0$ such that
$$
|\nabla u^+| \leq C\dfrac{u^+(x_0-re_n,t_0)}{r} \,\,\hbox{ and }\,\,
|\nabla u^-| \leq C\dfrac{u^-(x_0+re_n,t_0)}{r}
 $$
 in
$B_r(x_0)\times [t_0+t(x_0,r)/2, t_0+t(x_0,r)].$
\end{lemma}

\begin{proof}
The proof of this lemma is parallel to that of
Proposition~\ref{pri-reg-bad}. We use Harnack and backward Harnack
inequalities (Lemmas ~\ref{laterharnack} and \ref{backward}) instead
of Lemmas ~\ref{lem:harnack0} and ~\ref{backwardharnack0}. Also we
have Lemma~\ref{regularity}.
\end{proof}

\vspace{10 pt}

We are now ready to prove our main result, Proposition~\ref{main1}. Observe that (A) holds
up to some $T_0 >0$ by Lemma~\ref{measure} and
Lemma~\ref{lem:harnack0}.

\vspace{0.1 in}

\noindent {\it Proof of Proposition~\ref{main1}.}  Let $K$ be a
sufficiently large dimensional constant such that $K \gg M$.  Let us
assume that (A) breaks down for $u^+$ for the first time at $t=T_0$.
Then
\begin{equation}\label{contradiction}
\frac{u^{+}(z_0-s e_n, T_0)}{s} = KMC_0
\end{equation}
for some $(z_0,T_0) \in \Gamma \cap \Sigma$ and $r^{5/4}\leq s\leq
r$.
Let
\begin{equation}\label{definition00}
h = \sup\{h : \frac{u^{+}(z_0-k e_n, T_0)}{k} \geq M^2C_0 \hbox{ for } s \leq k \leq h \}.
\end{equation}
Note that $h<r/2$ due to Lemma~\ref{lem:harnack0} and the definition of $C_0$, and $h> 2s$ due
to Lemma~\ref{regularity}. By the definition of $h$ we have
\begin{equation}\label{condition0}
 \frac{u^{+}(z_0-h e_n, T_0)}{h} = M^2C_0.
 \end{equation}
Let us find $t_0$: the closest time before $T_0$ such that for some
$(y_0, t_0)\in \Gamma$
$$T_0-t_0 = t(y_0, h)/2 \,\,\hbox{ and }\,\, y_0/|y_0| =z_0/|z_0|. $$
Then Lemma~\ref{laterharnack} implies
$$
\frac{u^+(y_0-he_n,t_0)}{h} \sim \frac{u^+(y_0-he_n,T_0)}{h} \sim
\frac{u^+(z_0-he_n,T_0)}{h} =M^2C_0.
$$
Since $u^+(\cdot, t_0)$ and $u^-(\cdot, t_0)$ are comparable to
harmonic functions (Lemma~\ref{regularity}),
a similar argument as in
Lemma~\ref{bigbigno} implies that
$$
\dfrac{u^-(y_0+he_n,t_0)}{h} \lesssim C_0 \lesssim \frac{1}{M^2}
\frac{u^+(y_0-he_n,t_0)}{h}.
$$
Hence by Lemma~\ref{reg-bad}, we have
$$
|\nabla u^+(\cdot, T_0)| \sim M^2C_0 \hbox{ in } B_h(y_0)
$$
Since $B_s(z_0) \subset B_h(y_0)$, this would contradict
\eqref{contradiction} since $K \gg M$.

\hfill$\Box$

\section{Regularization after $t=t(r)$.}

Recall that $x_0 \in \Gamma_0$ and $r>0$ are fixed, and they satisfy
(\ref{est101}). Our goal is to prove the regularization of the free
boundary after the time $t(x_0, r)/2$ in $B_r(x_0)$. Define
$$
Q_r(x_0):=B_r(x_0) \times [t(x_0,r)/2, t(x_0,r)] \subset \Sigma.
$$
Let us briefly review the information we have on $u$ so far. Due to
Lemma 3.6 and Corollary 4.4,  our solution $u$ is $\e$-monotone in
$Q_r(x_0)$, with respect to a space and time cone, where the space
cone $W_x(e_n,\theta_0)$ satisfies $$|\theta_0-\pi| = O(L),$$ where
$L$ is the Lipschitz constant of the initial domain $\Omega_0$ given
by \eqref{initialLipschitz}. Moreover, due to Proposition 4.3, $u$
does not grow too big over time, which along with Lemma 3.8
guarantees that there is no big flux of $u$ coming in from outside
of $B_r(x_0)$ to perturb our solution.  Therefore the theory
developed in \cite{acs1}-\cite{acs2}, which says localized solutions
with flat free bondaries are smooth, applies with appropriate
modifications if we have $L$ small enough such that the waiting time
phenomena as seen in \cite{ck2} is prevented.  More precise
description of the situation as well as precise modifications are
detailed below.

 \vspace{10pt}

  As a result of
Proposition~\ref{main1}, (A) holds up to
$$
t=t(x_0,r) \leq
Cr^{2-\alpha} < r^{3/4}.
$$
Moreover $Q_r(x_0)\subset \Sigma$, and thus  Corollary~\ref{cor} and Lemma~\ref{star-shaped}, the free
boundary $\Gamma(u)$ is $r^{4/3}$-monotone in  $Q_r(x_0)$ with
respect to the time cone \\ $W_t(e_n, \tan^{-1}(1/K_1MC_0))$ and
 the space cone $W_x(e_n, \theta_0)$. Here $\theta_0$ is the
angle corresponding to the Lipschitz constant of $\Gamma_0$, and
$t(x_0, r) =\frac{r}{C_0}$.

\vspace{5pt}

On the other hand, by Lemma~\ref{lem:harnack0} and the definition of
$C_0$,
$$
\frac{u(x_0-re_n,\frac{t(x_0,r)}{2})}{C_0r} \sim 1.$$ Since
$Q_r(x_0) \subset \Sigma$, Proposition~\ref{main1} implies
$$
\dfrac{u(x,t)}{C_0r}\lesssim KM \,\,\hbox{ in }\,\,B_r(x_0) \times
[t(x_0,r)/2, t(x_0,r)].
$$
Motivated from the above estimates, we consider the re-scaled
function
$$
\tilde{u}(x,t):=\frac{1}{C_0r}u(rx+x_0,r^2t +\frac{t(x_0,r)}{2}).
$$

The main difficulty in applying the Method of [ACS]-[ACS2] lies in
the fact that  we cannot guarantee the $\e$-monotonicity of the
solution $u$ in time variable (although we can obtain, as above, the
$r^{4/3}$-monotonicity of the free boundary $\Gamma(u)$). In
[ACS]-[ACS2], it was important that initially  the time derivative
of the solution was assumed to be controlled by the spatial
derivative, i.e.,
\begin{equation}\label{00}
|u_t| \leq C(|Du^+|+|Du^-|).
\end{equation}
 Using \eqref{00} one can prove that the {\it direction vectors}
$$
\frac{Du^+}{|Du^+|}(-le_n,t), \quad \frac{Du^-}{|Du^-|}(le_n,t)
$$
 do not change  much for $ 0\leq t\leq l$. This is pivotal in regularization
 procedure since then $\Gamma(u)$ regularizes along the direction of
 the``common gain" obtained by those two direction vectors,
 the regularity of $\Gamma(u)$ then makes above two vectors line
 up better in a smaller scale, which contributes to further regularization
 of $\Gamma(u)$ in a finer scale.
 In our case we do not have \eqref{00}, which requires an extra care
in showing that the vectors do not change their directions too
rapidly.

\vspace{10pt}

$\circ$ {\it $\e$-monotonicity of $\Gamma(\tilde{u})$ to full
monotonicity of $\tilde{u}$}

\vspace{10pt}

First we prove that the $\e$-monotonicity of $\Gamma(\tilde{u})$
improves to Lipschitz continuity. Let $a=C_0r$. Then in the domain
$B_1(0)\times [-\frac{1}{a}, \frac{1}{a}]$, $\tilde{u}(x,t)$ solves
$$
\left\{\begin{array}{lll}
\tilde{u}_t-\Delta \tilde{u}=0 &\hbox{ in } &\{\tilde{u}>0\} \\ \\
V=a(|D\tilde{u}^+|-|D\tilde{u}^-|) &\hbox{ on }
&\partial\{\tilde{u}>0\}.
\end{array}\right.
$$
 Here note that $r^{7/6}\leq r^{\alpha}\leq a \leq
r^{\beta} \leq r^{5/6}$. In this scale, since $\tilde{u}$ is Caloric
and $\Gamma(\tilde{u})$ is $r^{1/3}$-close to a Lipschitz graph in
space and time, it follows that so does $\tilde{u}$ in
$B_{1/2}(0)\times [-\frac{1}{a}+1, \frac{1}{a}]$.

\vspace{5pt}

Note that in above step we are losing a lot of
information over time: $\Gamma(\tilde{u})$ is in fact
$r^{1/3}$-close to a Lipschitz graph moving very slow in time, but
this does not guarantee that $\tilde{u}$ also changes slowly in time.

\vspace{5pt}

We then follow the iteration process in Lemma 7.2 of [ACS]  to show the
following:
\begin{lemma}
If $r$ is sufficiently small, then  there exists $0<c,d<1/2$ such
that the following is true:
  $\tilde{u}$ is $\lambda r^{1/3}$-monotone in the cone of directions
$W_x(\theta_x-r^{d},e_n)$ and $W_t(\theta_t-r^{d},\nu)$ in the
domain $B_{1-r^{c}}(0)\times [\frac{(-1+r^c)}{a},\frac{1}{a}]$.

 \label{lem:flat}
\end{lemma}

One can then iterate above lemma to improve the $\e$-monotonicity to
full monotonicity, and state the result in terms of $\tilde{u}$:
\begin{lemma}
$\tilde{u}$ is fully monotone in $B_{1/2}(0)\times [0,\frac{1}{a}] $ for the cone
$$
\mathcal{C}_1:=W_x(\theta_x-r^{d},e_n)\cup W_t(\theta_t-r^{d},\nu),
$$
for some constant $0<d<1/2$.
\end{lemma}

\vspace{10pt}

 $\circ$ {\it Further regularity in space}

 \vspace{10pt}

Now we suppose $\tilde{u}$ is  Lipschitz in space and time. Then in
particular, we have the Lipschitz regularity of $u$ in space (and
very weak Lipschitz regularity of $u$ in time.)  We are interested
in proving the following type of statement:
\begin{lemma} [enlargement for the cone of monotonicity]
There exists $\lambda>0$ such that the following holds: Suppose
$\tilde{u}$ is Lipschitz with respect to the cone of monotonicity
$\Lambda_x(e_n,\theta_0)$ in $B_1(0)\times
[-\frac{1}{a},\frac{1}{a}]$. Then in the half domain
$B_{1/2}(0)\times [-\frac{1}{2a}, \frac{1}{2a}]$,  $\tilde{u}$ is
Lipschitz with respect to the cone of monotonicity $\Lambda_x(\nu,
(1+\lambda)\theta_0)$ with some unit vector $\nu$.
\end{lemma}

To prove the enlargement of the cone,  we take a closer look at the
change of $\tilde{u}$ over time, in the interior region.  More
precisely, we need the following lemma which follows  the approach
taken in \cite{cjk1} and \cite{cjk2}.

\begin{lemma}\label{upperbound101}
$$
|\tilde{u}_t| \leq a|D\tilde{u}|^2 \leq Ca \hbox{ in } [B_{1/2}(e_n) \cup
B_{1/2}(-e_n)] \times [-1/2a,1/2a],
$$
where $C$ is a dimensional constant.
\end{lemma}

\begin{proof}

1. The proof is similar to that of Lemma 8.3 of \cite{cjk2}. Note that
$\tilde{u}_t$ is a caloric function in $\Omega^+(\tilde{u})$ and
$\Omega^-(\tilde{u})$. Let us prove the lemma for $\tilde{u}^+$,
since parallel arguments apply to $\tilde{u}^-$.

2. We divide $\tilde{u}_t$ into two parts. More precisely, let
$$
\tilde{u}_t = v_1+v_2,
$$
where  both $v_1$ and $v_2$ are caloric in $\Omega^+(\tilde{u})$,
$v_1$ has initial data zero and the boundary data
$a|D\tilde{u}^+|(|D\tilde{u}^+|-|D\tilde{u}^-|)$ on
$\Gamma(\tilde{u})$, and $v_2$ has the initial data
$\tilde{u}_t(\cdot,-1/a)$  and the boundary data zero on
$\Gamma(\tilde{u})$.

3. As for $v_1$, we need to use the absolute continuity of the
caloric measure with respect to the harmonic measure, as well as the
Lipschitz continuity of the free boundary. we proceed as in Lemma
8.3 of \cite{cjk1}. Note that we have
$$
|D\tilde{u}^+| \sim |D\tilde{u}^-| \sim 1
$$
in   $[B_{1/2}(e_n) \cup B_{1/2}(-e_n)] \times [-1/a,1/a]$:  this
follows from the assumption (\ref{est101}), and
Lemmas~\ref{lem:harnack0} and \ref{backwardharnack0}. Therefore we
can proceed as in Lemma 8.3 of \cite{cjk1} to obtain
$$
v_1(x,t) \leq a\int _{\Gamma(\tilde{u})\cap\{-1/a\leq s\leq t\}}
|D\tilde{u}^+|^2 d\omega^{(x,t)} \leq a|D\tilde{u}|^2(x,t)
$$
where $\omega^{(x,t)}$ is the caloric measure for $\Omega(\tilde{u})$.

$$
v_1(x,t) \geq a\int _{\Gamma(\tilde{u})\cap\{-1/a\leq s\leq t\}}
-|D\tilde{u}^-|^2 d\omega^{(x,t)} \geq -a|D\tilde{u}|^2(x,t).
$$

4. As for $v_2$, we conclude that it must be smaller than that of
caloric function solved in the whole domain with the absolute value
of its initial data. The advantage is that then we can use the heat
kernel. Note that  the initial data is given at $t=-1/a$ and has a
compact support. The initial data is given by $v_t \leq
\frac{C}{a}v_{e_n}$, where $v_{e_n}(x,t)$ is comparable to the
derivative of harmonic function in Lipschitz domain.

Therefore the heat kernel representation is given as
$$
\dfrac{1}{(t+1/a)^{\frac{n}{2}+1}}\int |x_n-y_n|\exp^{-|x-y|^2/(t+1/a)} v(y,-1/a) dy.
$$
Since $t\in [0,1/a]$, and $k\exp^{-a k^2} \leq C\exp^{-\frac{a}{2}k^2}$,  we get the effect of $O(a)$.

\end{proof}

Now we change the scale, and consider the function

\begin{equation}\label{definition101}
v(x,t):= \frac{1}{C_0r}u(rx+x_0, \frac{r}{C_0}t+1)
\end{equation}

Then this function is Lipschitz continuous, in space and time, {\it
away } from the free boundary. The following lemma suggests that the
cone of monotonicity improves away from the free boundary, as we
look at smaller scales. The proof is parallel to that of Lemma 8.4  in \cite{acs2}.

\begin{lemma}
  Let $v$ given by \eqref{definition101}. Suppose that there
 exists constants $\delta>0$ and $0\leq A \leq B, \mu:= B-A$ such that
 $$
 \alpha(Dv, -e_n) \leq \delta \hbox{ and } A \leq \dfrac{v_t}{-e_n\cdot Dv} \leq B
 $$
in $B_{1/6}(-\frac{3}{4}e_n)\times (-\delta/\mu, \delta/\mu)$ with
$\frac{\delta}{\mu}<r$. Then there exist a unit vector $\nu\in\R^n$
and positive constants $r_0, b_0<1$ depending only on $A$, $B$ and
$n$ such that
$$
\alpha(Dv(x,t),\nu) \leq b_0\delta \hbox{ in }
B_{1/8}(-\frac{3}{4}e_n) \times (-r_0\frac{\delta}{\mu},
r_0\frac{\delta}{\mu}).
$$

\end{lemma}

Now we can proceed as in section 6 of \cite{cjk2} to obtain further regularity, using
Lemma~\ref{upperbound101} instead of the uniform upper bound on
$|Du|$ up to the free boundary.

\begin{theorem}
$\Gamma(v)$ is $C^1$ in space in $Q_{1/2}$. In particular,
three exist constants $l_0,C_0>0$ depending only on $L,n$ and $M$
such that for a free boundary point $(x_0,t_0)\in
\Gamma(v)$, $\Gamma(v)\cap B_{2^{-l}}(x_0,t_0)$ is a
Lipschitz graph with Lipschitz constant less than $\dfrac{C_0}{l}$
if $l\geq l_0$.
\end{theorem}

\vspace{10pt}

$\circ$ {\it Regularity in time}

\vspace{10pt}

Lastly, proceeding as in section 7-8 of \cite{cjk2} yields the
differentiability of $\Gamma(v)$ in time. The main step in the
argument is the following proposition: the statement and its proof
is parallel to those of Theorem 7.2 in \cite{cjk2}.

\begin{proposition}
There exist constants $l_0>0$ and $1<\gamma<2$ depending only on
$L,n,M$ such that for $(x_0,t_0)\in\Gamma(v)\cap Q_1$, if
$l>l_0$ then $\Gamma(v)\cap B_{2^{-l}}(x_0,t_0)$ is a
Lipschitz graph with Lipschitz constant less than $l^{-\gamma}$.
\end{proposition}

Above proposition and the blow-up argument in section 8 of \cite{cjk2}
yields the desired result:

\begin{theorem}\label{thm:last}
$\Gamma(v)$ is differentiable in time. Moreover
$$
C^{-1} \leq |Dv|(x,t)\leq C\hbox{ in } \Omega(\tilde{u})\cap
Q_{1/2},
$$
where $C=C(M,n)$.
\end{theorem}

\section{General case: solutions with Locally Lipschitz Initial data}

In this section, we present how to extend the result of the main
theorem to solutions with locally Lipschitz initial data. Our
setting is as follows. Suppose $\Omega_0$ is a bounded region in
$B_R(0)$. Suppose $u$ is a solution of (ST2) with $u_0\geq -1$,
$u_0=-1$ on $B_R(0)$ and $u_0 \leq M_0$. Further suppose that
$\Omega_0$ is locally Lipschitz: that is, for any $x_0\in\Gamma_0$,
$\Gamma_0 \cap B_1(x_0)$ is Lipschitz with a Lipschitz constant $L
\leq L_n$.

\begin{figure}
\center{\epsfig{file=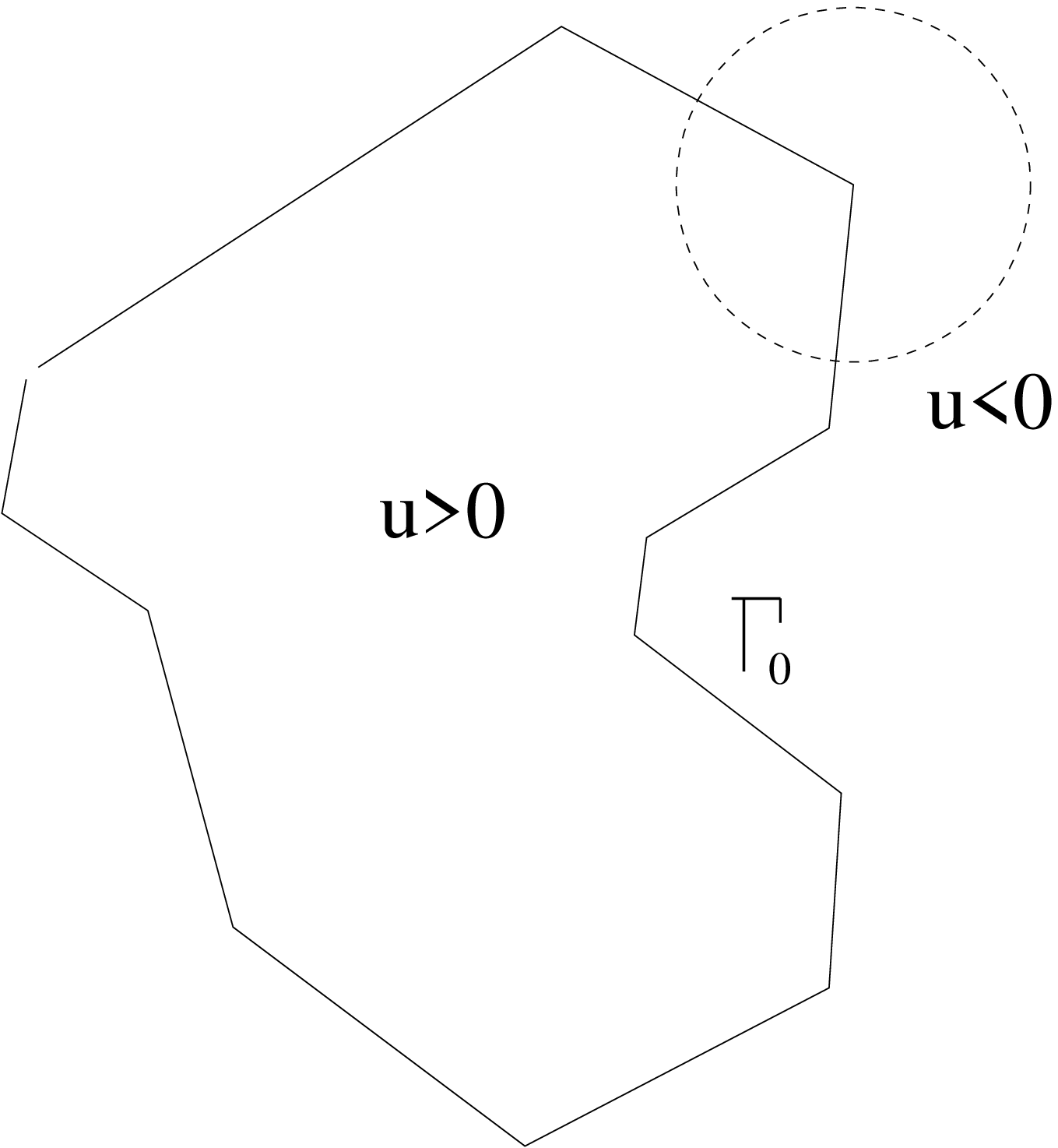,height=2.0in}} \center{Figure
4: Locally Lipschitz initial domain}
\end{figure}

Let the initial data $u_0$ solve  $\Delta u_0 =0$ in $B_1(x_0)$. Then we claim that the
parallel statements as in Theorem~\ref{maintheorem} hold in
$B_{2d_0}(x_0) \times [t(x_0,d_0)/2, t(x_0, d_0)]$, where $d_0$ is a
constant depending on $n$ and $M_0$. More precisely:

\begin{theorem}\label{maintheorem2}
Suppose $u$ is a solution of (ST2) with initial data $u_0$ such that
$-1 \leq u_0 \leq M_0$. Further suppose that for $x_0 \in \Gamma_0$,
$\Gamma_0 \cap B_1(x_0)$ is Lipschitz with a Lipschitz constant $L
\leq L_n$ and $\Delta u_0 =0$ in the positive and negative phases of
$u_0$ in $B_1(x_0)$. Then there exists a constant $d_0>0$ depending
on
 $n$ and $M_0$ such that (a) and (b) of
Theorem~\ref{maintheorem} hold for $u$ and $d \leq d_0$.

\end{theorem}

The proof of the above theorem is parallel to that
of Theorem~\ref{maintheorem} in section 5, after proving the
following lemma.

\begin{lemma}
There exists a solution $v$ of (ST2) with a star-shaped initial data
such that the level sets of $u$ and $v$ are $\e d_0$-close to each
other in $B_{2d_0}(x_0)$ up to the time $t(x_0,d_0;u)$, where
$d_0>0$ is sufficiently small. In particular, $u$ and $\Gamma(u)$ is
$\e$-monotone in a cone of $W_x$ and $W_t$ in $B_{2d_0}(x_0) \times
[t(x_0, d_0)/2,t(x_0,d_0)]$.
\end{lemma}

Even though our equation is nonlocal, the behavior of far-away
region would not affect much the behavior of solution in the unit
ball, if the solution behaves ``reasonably'' outside the unit
ball. For example, in the star-shaped case, we know at least that
the free boundary is almost locally Lipschitz at each time. In the
locally Lipschitz case, we control the solution by putting an
upper bound $M_0$ on the initial data $u_0$. We will argue that in
a sufficiently small subregion of $B_1(x_0) \times [0,1]$, the
solution is mostly determined by the local initial data in
$B_1(x_0)$. The perturbation method in the proof of Lemma 2.4 in
\cite{cjk1} will be adopted here.  Denote  $B_1(x_0) =B_1$.

\vspace{5pt}

1. Construct a star-shaped region $\Omega'\subset B_R(0)$ such
that
\begin{itemize}
\item[(a)] $\Omega'\cap B_1 = \Omega_0\cap B_1$.
\item[(b)]$\Omega'$ is star-shaped with respect to every $x\in
K\subset\Omega'$ for a sufficiently large ball $K$.
\end{itemize}
Let $v_0^+$ be the harmonic function in $\Omega'-K$ with boundary
data $1$ on $\partial K$, and $0$ on $\partial\Omega'$. Next, let
$v_0^-$ be the harmonic function in $B_R(0)-\Omega'$ with boundary
data $1$ on $\partial B_R(0)$, and $0$ on $\partial\Omega'$. Let
$B_2$ be a concentric ball in $B_1$ with the radius of $\e^{k_0}$,
i.e.,
$$
B_2=B_{\e^{k_0}}(x_0) \subset B_1(x_0)=B_1.
$$
 Let $k_0$ be sufficiently large. Then by Lemma~\ref{lem:JK},
a normalization of $v_0^\pm$ by a suitable constant multiple
yields that for any $x\in B_2$
\begin{equation}\label{initial}
1-\e \leq \dfrac{u_0(x)}{v_0(x)} \leq 1+\e.
\end{equation}
 Let $v$ solve (ST2) with
initial data $v_0=v_0^+ -v_0^-$. Then Theorem~\ref{maintheorem}
applies for $v$ since $v_0$ is star-shaped with respect to $K$.

 For the proof of the claim, we will
find a sufficiently small $d_0$ such that $v$ is $\e d_0$-close to
$u$ in $B_{2d_0}(x_0)$ up to the time $t(x_0,d_0)$. More
precisely, we will construct a supersolution $w_1$ and a
subsolution $w_2$ of (ST2) such that in some small ball
$B_h(x_0)$, we have
 $$
 w_2\leq u\leq w_1
 $$
and the level sets of $w_1$ and $w_2$ are $h\e$ close to the level
sets of $v$.

\medskip

2. Let $k_1$ and $k_2$ be large constants which will be determined
later. Define
$$
 H^{\pm}:= (\Gamma_0(v) \pm
\e^{k_0+k_1} e_n)\cap B_2.
$$
Let
$$
d_0:= \e^{k_0+k_1+k_2}.
$$
and let $t(d_0):= t(x_0, d_0; v)=t(x_0, d_0; u)$.
  First note that
$$t(d_0) \geq d_0^{2-\beta}\geq \e^{7(k_0+k_1+k_2)/6}.$$
 Hence for $v$ to be almost
harmonic in a scale much larger than $\e^{k_0+k_1}$, we need
$\sqrt{t(d_0)} >\e^{k_0}$, i.e.,
$$
 7(k_0+k_1+k_2)/12 < k_0.
$$
Observe that by the construction of $H^{\pm}$ and $d_0$,
\begin{equation} \label{localization}
\sqrt{ t(d_0)} \gg {\rm radius}(B_2) \gg
 {\rm dist}(H^{\pm}, \Gamma_0) \gg  \max_{x \in \Gamma_{t} \cap B_2,0\leq t\leq t(d_0) } {\rm dist}(x, \Gamma_0)
\end{equation}
where the last inequality follows from Lemma~\ref{good} if we choose
$k_2 \geq 2k_1$. If $k_2$ is  sufficiently large, then one can prove
from the last inequality of (\ref{localization}) and the bound on
$v_t$ that
\begin{equation}\label{claim001}
1 -\e \leq \dfrac{|v(x,t)|}{|v_0(x)|} =
\dfrac{|v(x,t)|}{|u_0(x)|}\leq 1+\e \hbox{ on } H^{\pm} \times
[0,t(d_0)].
\end{equation}

\medskip

3. We do have an estimate, Lemma~\ref{good}, on how far the
boundaries move away for the local one-phase case. If we take the
one-phase versions with initial data $u_0^+$ and $u_0^-$, and
compare  with $u$, then  we obtain that   $\Gamma(u) \cap B_2$ stays
in the $d_0^{\frac{2-\alpha}{2-\beta}}$-neighborhood of
$\Gamma_0(u)\cap B_2$ up to the time $t(d_0)=t(x_0, d_0)$. In other
words, the free boundary of $u$ moves less than $d_0^{5/7}$ in $B_2$
up to the time $t(d_0)$.

\vspace{10pt}

Now we let $S$ be the region between $H^+$
and $H^-$. To construct a sub (or super) solution in S, we take the
fixed boundary data $(1-\e)v_0(x)$  on $H^{-}$ (or $H^{+}$), and
$(1+\e)v_0(x)$  on $H^{+}$ (or $H^{-}$). To control the effect from
the side $\partial B_2 \cap S$, we bend the free boundary
$\Gamma_t(v)$ by $d_0^{5/7}$ on each side of $\partial B_2 \cap S$,
using the conformal mapping $\hat{\Phi}$ (or $\breve{\Phi}$). (See
section 4 of for the definition of $\hat{\Phi}$ and $\breve{\Phi}$.)
More precisely, we bend the free boundary of $v$ downward (or
upward) using the conformal map $\hat{\Phi}$ (or $\breve{\Phi}$),
and solve the heat equation in there. Then similar arguments as in
Lemmas 4.1 and 4.3 of \cite{ck} yield that the solution is still
(almost) a supersolution, and it stays close to the original
solution.


\begin{thebibliography}{[MT]}
\bibitem[ACF]{acf84} H. W. Alt, L. A. Caffarelli, and A. Friedman, {Variational problems with two phases and their free boundaries,} {\em{Trans. Amer. Math. Soc.}} {\bf 282} (1984), no. 2, pp 431-- 461.
\bibitem[ACS1]{acs1} I. Athanasopoulos, L. Caffarelli and S. Salsa, { Caloric functions
in Lipschitz domains and the Regularity of Solutions to Phase
Transition Problems,} {\em{Ann. of Math. (2)}}{\bf 143} (1996) pp
413--434.
\bibitem[ACS2] {acs2} I. Athanasopoulos, L. Cafferlli and S. Salsa, {Phase Transition problems of parabolic type: flat
free boundaries are smooth,} {\em{ Comm. Pure Appl. Math.}}{\bf 51}
(1998) pp 77--112.
\bibitem[ACS3]{acs3} I. Athanasopoulos, L. Cafferlli and S. Salsa,
{Regularity of the free boundary in parabolic phase-transition
problems.} {\em{Acta Math.}} {\bf 176} (1996), pp 245--282.
\bibitem[C1]{c1} L. Caffarelli, {The regularity of free boundaries in higher dimensions,}{\em{Acta Math}}{\bf 139} (1977), pp 155-184.
\bibitem[C2]{c2} L. Caffarelli, { A Harnack inequality approach to the regularity of
free boundaries, Part I: Lipschitz free boundaries are
$C^{1,\alpha}$,} {\em{Rev. Mat. Iberoamericana}}{\bf 3} (1987), no.
2, 139--162.
\bibitem[C3]{caf93} L. A. Caffarelli, { A monotonicity formula for heat functions in disjoint domains, Boundary value problems for partial differential equations and applications, } Masson, Paris, 1993, pp. pp 53--60.
\bibitem[C-C] {c-c} L.A. Caffarelli, X. Cabre, {\em{Fully Nonlinear
Elliptic Equations,}}
 Amer. Math. Soc.,  colloquium publication, 43, Providence, R.I., 1995
\bibitem[CPS]{cps} L. Caffarelli, A. Petrosyan and H.
Shahgholian,{ Regularity of a free boundary in parabolic potential
theory,} Journal of AMS {\bf 17} (2004) pp 827--869.
 \bibitem[CS]{cs} L. Caffarelli, S. Salsa, {A geometric approach to free boundary problems,} Graduate studies in mathematics, AMS (2005).\bibitem[CJK1]{cjk1} S. Choi, D. S. Jerison and I. C. Kim, { Regularity for the One-Phase
Hele-Shaw problem from a Lipschitz initial surface,} {\em{Amer. J.
Math.}} {\bf 129} (2007) pp 527--582.
\bibitem[CJK2]{cjk2} S. Choi, D. S. Jerison and I. C. Kim, { A local regularization theorem on one-phase Hele-Shaw
flow,}{\em{ Indiana U. Math. Journal, }}{\bf 58 }(2009) pp. 2765-2804
\bibitem[CK]{ck} S. Choi and I. Kim, {Regularity of one-phase Stefan problem near Lipschitz initial domain,} to appear in {\em{Amer. J. Math.}} (2010).
\bibitem[CK2]{ck2} S. Choi and I. Kim, { Waiting time phenomena of the Hele-Shaw and the Stefan problem,}{\em{ Indiana U. Math. Journal,}} {\bf 55} (2006) pp 525-552.

\bibitem[D]{d} B. Dahlberg, {\em{ Harmonic functions in Lipschitz domains, }} Harmonic analysis in Euclidean spaces,
Part 1, pp. 313--322, Proc. Sympos. Pure Math., XXXV, Part, Amer.
Math. Soc., Providence, R.I., 1979
\bibitem[FGS1]{fgs1} E. B. Fabes, N. Garofalo and S. Salsa, {Comparison Theorems for Temperatures in non-cylindrical Domains,}
{\em{Atti Accad. Naz. Lincei, Read. Ser. 8,}}{\bf 78} (1984) pp.
1--12
\bibitem[FGS2]{fgs2} E. B. Fabes, N. Garofalo and S. Salsa, {\em{Illinois Journal of Mathematics}}{\bf 30}
(1986) pp 536--565.

\bibitem[GZ]{gz} I. G. G\"{o}tz and B. Zalzman, {Nonincrease of mushy region in a nonhomogeneous Stefan problem,}
{\em Quart. Appl. Math.} {\bf XLIX} (1991), no. 4, 741Ð746.

\bibitem[JK]{jk} D. S. Jerison and C. E. Kenig, { Boundary behavior of Harmonic functions in Non-tangentially
Accessible Domains, } {\em Advan. in Math.} {\bf 46} (1982), 80-147.
\bibitem[K]{k} I. C. Kim, {\em Uniqueness and
Existence result of Hele-Shaw and Stefan problem,} Arch. Rat. Mech.
Anal,  {\bf 168} (2003), 299-328.
\bibitem[KP]{kp} I.C. Kim and N. Pozar, {Viscosity solutions for the two-phase Stefan problem, } to appear in Comm. PDE.
\bibitem[King]{ki} J. R. King, {Development of singularities in some moving boundary problems,} European J. Appl. Math. {\bf 6} (1995)
\bibitem[M]{m} A. M. Meirmanov, {The Stefan Problem,} de Gruyter Expositions in Mathematics, vol. 3, Walter de Gruyter \& Co., Berlin, 1992, ISBN
\bibitem[OPR]{opr} O.A. Oleinik, M. Primicerio, and E. V. Radkevich, {Stefan-like problems, } Meccanica {\bf 28} (1992), 129-143.
\bibitem[RB]{rb} W. Rogers and A. E. Berger, {Some properties of the nonlinear semigroup for the problem $u_t-Df(u)=0$,} Nonlinear Anal., Theory, Methods and Applications {\bf 8} (1984), no. 8, 909-939.
\bibitem[R]{r} L. Rubinstein, {The Stefan Prolem,} American Mathematical Society, Providence, R.I., 1971.
\end{thebibliography}
\end{document}